\documentclass[11pt]{amsart}
\usepackage[utf8]{inputenc}
\usepackage{amsmath}
\usepackage[makeroom]{cancel}
\usepackage{tikz-cd} 
\usepackage{todonotes}
\usepackage{amsfonts, amsthm}
\usepackage{subfig}
\usepackage{mathtools}
\usepackage{amssymb, hyperref, color, graphicx, caption}
\usepackage{geometry}
\usepackage{mdframed}
\usepackage{tikzsymbols}
\usepackage{tcolorbox}
\usepackage{xcolor}   
\newcommand{\Z}{\mathbb{Z}}

\newcommand{\C}{\mathbb{C}}

\newcommand{\F}{\mathbb{F}}
\newcommand{\R}{\mathbb{R}}

\newcommand{\T}{\mathbb{T}}
\newcommand{\ET}{\widetilde{R}}
\newcommand{\MT}{\mathrm{MT}}

\newtheorem{thm}{Theorem}

\numberwithin{thm}{section}
\newtheorem*{xthm}{Theorem}
\newtheorem{cor}[thm]{Corollary}
\newtheorem{lem}[thm]{Lemma}

\newtheorem{prop}[thm]{Proposition}
\theoremstyle{definition}
\newtheorem{defn}[thm]{Definition}

\newtheorem{rem}[thm]{Remark}

\makeatletter
\newtheorem*{rep@theorem}{\rep@title}
\newcommand{\newreptheorem}[2]{%
\newenvironment{rep#1}[1]{%
 \def\rep@title{#2 \ref{##1}}%
 \begin{rep@theorem}}%
 {\end{rep@theorem}}}
\makeatother

\newreptheorem{theorem}{Theorem}

\newreptheorem{lemma}{Lemma}

\begin{document}

\title[S. Das, H. Lang, H. Wan, and N. Xu]{The Distribution of $k$-Free Effective Divisors and the Summatory Totient Function in Function Fields}

\author[S.~Das]{Sanjana Das}
\address{Department of Mathematics \\ 
Massachusetts Institute of Technology \\ 
77 Massachusetts Avenue \\
Cambridge, MA 02139}
\email{\href{mailto:sanjanad@mit.edu}{sanjanad@mit.edu}}

\author[H.~Lang]{Hannah Lang}
\address{Department of Mathematics \\
Harvard University \\
1 Oxford Street \\
Cambridge, MA 02138}
\email{\href{mailto:hlang@college.harvard.edu}{hlang@college.harvard.edu}}

\author[H.~Wan]{Hamilton Wan}
\address{Department of Mathematics \\
Yale University\\
10 Hillhouse Avenue\\
New Haven, CT 06520}
\email{\href{mailto:hamilton.wan@yale.edu}{hamilton.wan@yale.edu}}

\author[N.~Xu]{Nancy Xu}
\address{Department of Mathematics \\
Princeton University \\
304 Washington Road \\
Princeton, NJ 08544}
\email{\href{mailto:nancyx@princeton.edu}{nancyx@princeton.edu}}

\maketitle

\begin{abstract}
    Motivated by the study of the summatory $k$-free indicator and totient functions in the classical setting, we investigate their function field analogues. First, we derive an expression for the error terms of the summatory functions in terms of the zeros of the associated zeta function. Under the Linear Independence hypothesis, we explicitly construct the limiting distributions of these error terms and compute the frequency with which they occur in an interval $[-\beta, \beta]$ for a real $\beta > 0$. We also show that these error terms are unbiased, that is, they are positive and negative equally often. Finally, we examine the average behavior of these error terms across families of hyperelliptic curves of fixed genus. We obtain these results by following a general framework initiated by Cha and Humphries. 
\end{abstract}

\section{Introduction}\label{sec:intro}
\subsection{Distribution of \texorpdfstring{$k$}{k}-Free Numbers in the Classical Setting}

For positive integers $n$, consider the $k$-free indicator function \[\mu_k(n) := \begin{cases} 1 &  \text{if $n$ is not divisible by $p^k$ for any prime $p$} \\ 0 & \text{otherwise}.\end{cases} \] One can study the distribution of $k$-free numbers by examining the summatory function \[Q_k^*(x) := \sum_{n \leq x} \mu_k(n).\] Classical arguments show that \[Q_k^*(x) = \frac{x}{\zeta(k)} + R^*_k(x),\] where $\zeta(s)$ is the Riemann zeta function and $R_k^*(x)$ is the error term. The problem of understanding the limiting behavior of $R^*_k(x)$ as $x$ tends to infinity is quite well-studied. Unconditionally, Walfisz \cite{Wal63} and Evelyn and Linfoot \cite{EL31} have produced upper and lower bounds on $R^*_k(x)$, respectively. However, many results of this flavor are conditional, obtained by writing $R_k^*(x)$ as a sum in terms of zeros of the Riemann zeta function and making plausible assumptions about the behavior of these zeros. For example, by assuming the Riemann hypothesis, Montgomery and Vaughan \cite{MV81} derived an upper bound of \[R^*_k(x) = O_\varepsilon(x^{1/(k+1)+\varepsilon}).\] Moreover, Meng \cite{Men17} showed that $x^{-\frac{1}{2k}}R^*_k(x)$ admits a limiting logarithmic distribution under the Riemann hypothesis and several other assumptions; by assuming the linear independence of the ordinates of zeros of $\zeta(s)$, he also explicitly computed the Fourier transform of the limiting distribution and studied the order of growth of the tails of this distribution. On the other hand, recent work of Mossinghoff, Oliviera e Silva, and Trudgian \cite{MOT21} shows that one cannot do better than $O(x^{1/2k})$. However, the conjectured bound \[R^*_k(x) = O_\varepsilon(x^{1/2k + \varepsilon}) \] remains unproven \cite{Men17}, even under the Riemann hypothesis. 

\subsection{Distribution of \texorpdfstring{$k$}{k}-Free Divisors in Function Fields}

Motivated by the study of $k$-free numbers in the classical setting, we study the analogues of these problems in function fields. The function field setting provides a natural variant of these problems, with the advantage that many results in the function field setting are unconditional since the Riemann hypothesis has been proven. Our results and techniques closely follow the work of Humphries \cite{Hum12, Hum13} and Cha \cite{Cha17}, who studied the function field analogue of the summatory M\"obius function and its limiting distribution. To set up the function field analogue of the $k$-free indicator function, fix $q = p^n$ for some prime number $p$ and positive integer $n$, and let $C$ be a nonsingular projective curve over the finite field $\F_q$ of genus $g$ with associated function field $C/\F_q$. Define the $k$-free indicator function $\mu_k$ on the effective divisors of $C/\F_q$ by \[\mu_k(D) := \begin{cases} 1 & \text{if $D- kP$ is not an effective divisor for any prime divisor $P$} \\ 0 & \text{if $D - kP$ is an effective divisor for some prime divisor $P$.} \end{cases}\] 
As in the classical case, we study the limiting behavior of the summatory function \[Q_k(X) := \sum_{\deg(D) < X} \mu_k(D)\] as $X$ approaches infinity. Note that $Q_k(X)$ counts the $k$-free effective divisors with degree less than $X$. 

Before proceeding to the statement of our results, we establish some preliminaries on the \emph{zeta function} associated to $C/\F_q$, which forms the central object of study in this article. If we let $\mathcal{N}D$ denote the norm of any effective divisor $D$, then the zeta function is defined as \[\zeta_{C/\F_q}(s) := \sum_{D \geq 0} \frac{1}{\mathcal{N}D^s} = \prod_{P \text{ prime}} \left(1 - \frac{1}{\mathcal{N}P^s}\right)^{-1},\] where the sum and product converge absolutely when $\Re(s) > 1$. Note that each $\mathcal{N}D$ is a power of $q$, so we can alternatively express $\zeta_{C/\F_q}(s)$ as a function in the variable $u := q^{-s}$, which we denote by $Z_{C/\F_q}(u) := \zeta_{C/\F_q}(s)$. Moreover, one can show that $Z_{C/\F_q}(u)$ has the form \[Z_{C/\F_q}(u) := \frac{L_{C/\F_q}(u)}{(1-qu)(1-u)},\] where $L_{C/\F_q}(u) \in \Z[u]$ is a polynomial of degree $2g$ such that $L_{C/\mathbb{F}_q}(0) = 1$,  $L_{C/\mathbb{F}_q}(1) = h_{C/\mathbb{F}_q}$ (here $h_{C/\mathbb{F}_q}$ denotes the class number of $C/\mathbb{F}_q$), and $L_{C/\mathbb{F}_q}(q^{-1}u^{-1}) = q^{-g}u^{-2g}L_{C/\mathbb{F}_q}(u)$ \cite{Ros02}. This expression gives $Z_{C/\F_q}(u)$ an analytic continuation to $\C$ with simple poles at $u = 1$ and $u = q^{-1}$. Moreover, after factoring \[L_{C/\F_q}(u) := \prod_{j=1}^{2g} (1 - \gamma_ju),\] we refer to the complex numbers $\gamma_j$ as the \emph{inverse zeros} of $Z_{C/\F_q}(u)$, and use $\theta(\gamma_j)$ to denote the argument of $\gamma_j$. These inverse zeros come in conjugate pairs, and we index them so that $\gamma_{j+g} = \overline{\gamma_j}$ for $j= 1,\ldots,g$, and $0 \leq \theta(\gamma_j) \leq \pi$ for $j = 1, \ldots, g$. All of the results in this article will rely on expressing summatory functions in terms of these inverse zeros. The function field case differs crucially from the classical setting in that the zeta function has only finitely many zeros (in particular, $2g$). In addition, the Riemann hypothesis for function fields has been proven, so the absolute values of these inverse zeros are known to be $\sqrt{q}$.

\begin{xthm}[Riemann hypothesis for Function Fields]
    The zeros of $\zeta_{C/\F_q}(s)$ lie on the line $\Re(s) = 1/2$. Equivalently, the inverse zeros $\gamma_j$ all satisfy $|\gamma_j| = \sqrt{q}$. 
\end{xthm}

We say that $C/\F_q$ satisfies the \emph{Linear Independence hypothesis} if the set \[\left\{\theta(\gamma_1),\ldots, \theta(\gamma_{g})\right\} \cup \{\pi\}\] is linearly independent over $\mathbb{Q}$.

The assumption of Linear Independence will play an important role when we compute limiting distributions of error terms and study the global behavior of these error terms over families of hyperelliptic curves. Having established a few basic properties of the zeta function, we return to the problem of finding the distribution of $k$-free effective divisors in the function field $C/\F_q$.

Say that the zeta function $Z_{C/\F_q}(u)$ has inverse zeros $\gamma_1,\ldots,\gamma_{2g}$. For each $j = 1,\ldots,2g$ and $\ell = 0,\ldots,k-1$, let $\gamma_{j, \ell} := q^{1/2k}e^{i(\theta(\gamma_j) + 2\pi \ell)/k}$. Under the assumption that the zeros of $Z_{C/\F_q}(u)$ are simple, we show that the explicit expression for the normalized error term \[\ET_k(X) := \frac{Q_k(X) - \MT_k(X)}{q^{X/2k}}\] is given by 
\begin{align} \label{eqn:k_free_eq}
    \ET_k(X) = -\sum_{j=1}^{2g}\sum_{\ell=0}^{k-1} \frac{Z_{C/\F_q}(\gamma_{j,\ell}^{-1})}{k\gamma_{j,\ell}^{1-k}Z'_{C/\F_q}(\gamma_j^{-1})}\frac{\gamma_{j,\ell}}{\gamma_{j,\ell}-1}e^{iX(\theta(\gamma_j) + 2\pi i\ell)/k} + O_{q,g}\left(q^{-X/2k}\right), 
\end{align} where the main term $\MT_k(X)$ is
\[
\MT_k(X) := \frac{q^{1-g}h_{C/\F_q}}{\zeta_{C/\F_q}(k)(q-1)^2}q^X,
\] and $h_{C/\F_q}$ is the class number of $C/\F_q$. In particular, note that the normalized error term $\ET_k(X)$ is $O(1)$ as $X$ approaches infinity, since all the summands have fixed absolute value. 

When the zeta function $Z_{C/\F_q}(u)$ has non-simple zeros, rather than providing an explicit formula, we instead study the order of growth of $R_k(X) := Q_k(X) - \MT_k(X) $ by proving 
\[R_k(X) = O(X^{r-1}q^{X/2}) \quad \textnormal{and}\quad\limsup_{X \to \infty}\frac{R_k(X)}{X^{r-1}q^{X/2}} > 0,\] where $r$ is the maximal order of a zero of $Z_{C/\F_q}(u)$. 

Assuming the Linear Independence hypothesis, we use (\ref{eqn:k_free_eq}) to quantify the frequency with which $\ET_k(X)$ occurs an interval $[-\beta, \beta]$ for a real $\beta > 0$ by finding a limiting distribution for $\ET_k(X)$. In the classical setting, Meng \cite{Men17} proved the existence of the limiting logarithmic distribution of the error term for $k \geq 2$, i.e., there exists a probability measure $\nu_k^*$ defined on the Borel subsets of $\R$ such that
\[
\lim_{Y \to \infty} \frac{1}{Y} \int_{0}^{Y} f(e^{-y/2k} R^*_k(e^y)) \, dy = \int_{\mathbb{R}} f(x) \, d \nu_k^* (x)
\]
for all bounded Lipschitz continuous functions $f:\mathbb{R} \rightarrow \mathbb{R}$. Analogously, we will prove that the function field error term $\ET_k(X)$ also has a limiting distribution $\nu_k$: 
\[\lim_{r \to \infty} \frac{1}{r}\sum_{X=1}^r f(\ET_k(X)) = \int_\R f(x) \, d\nu_k(x)\] for all bounded continuous functions $f: \mathbb{R} \rightarrow \mathbb{R}$. We note the fact that the limiting \textit{logarithmic} distribution is studied in the classical case, in contrast to our result in the function field setting. Meng \cite{Men17} proved the existence of limiting distributions by showing that $e^{-y/2k} R_k^*(e^y)$ is a $B^2$-almost periodic function, a class of functions that are known to possess limiting distributions \cite{ANS14}. The approach to proving the existence of a limiting distribution in the function field setting is far simpler --- zeta functions of function fields have only finitely many zeros, and accordingly, the normalized error term $\ET_k(X)$ is expressed as only a finite sum (up to a negligibly small error term), allowing for the straightforward application of tools from probability theory, as done in \cite{Cha08} and \cite{Hum12}. Using this limiting distribution, we can compute the natural density of the set \[\mathcal{S}_k(\beta) := \{X \in \Z^+ \ |\ |\ET_k(X)| \leq \beta\}\] to quantify how often $\ET_k(X)$ is bounded by any $\beta > 0$. For a Borel set $B \subset \mathbb{R}$ and a Borel-measurable function $f: [0, 2 \pi)^g \rightarrow \mathbb{R}$, we write $m(f(\theta_1, \hdots, \theta_g) \in B)$ for $$m(\{(\theta_1, \hdots, \theta_g) \in [0, 2 \pi)^g \ |\ f(\theta_1, \hdots, \theta_g) \in B\}),$$ where $m$ is the Lebesgue measure on $[0, 2 \pi)^g$.

\begin{thm}\label{intro_thm:kfree_limdist}
Let $C/\mathbb{F}_q$ be a function field with genus $g \geq 1$, and suppose that $C$ satisfies the Linear Independence hypothesis. For a real $\beta > 0$, the natural density $$\delta(\mathcal{S}_k(\beta)) = \lim_{Y \to \infty} \frac{1}{Y} \# \{1 \leq X \leq Y \ |\ |\ET_k(X)| \leq \beta\}$$ exists and is given by 
\[\delta(\mathcal{S}_k(\beta)) = \sum_{t=0}^{k-1} \frac{1}{k} \cdot m\left(\sum_{j=1}^g 2\left|\sigma_{t,j}^{(k)}\right|\cos(\theta_{j}) \in [-\beta,\beta]\right),\] where \[\sigma_{t,j}^{(k)} := \sum_{\ell=0}^{k-1}\frac{Z(\gamma_{j,\ell}^{-1})}{k\gamma_{j,\ell}^{1-k}Z'(\gamma_j^{-1})}\frac{\gamma_{j,\ell}}{\gamma_{j,\ell}-1}e^{2\pi i\ell t/k}\] for each $j = 1,\ldots,g$ and $t =0,1,\ldots,k-1$. 
\end{thm}

The above theorem is analogous to \cite[Theorem 1.6]{Hum12} in that both expressions for the natural density are given in terms of the zeros of the associated zeta function. However, the application of the Kronecker--Weyl theorem in the $k$-free indicator function setting is further complicated by the presence of multiple $k$th roots whose arguments are not linearly independent.

By exploiting the translation invariance of the Lebesgue measure, we find that the normalized error term $\ET_k(X)$ is unbiased, i.e., it is positive and negative equally often.

\begin{cor}\label{intro_cor:kfree_no_bias}
    Let $C/\mathbb{F}_q$ be a function field with genus $g \geq 1$, and suppose that $C$ satisfies the Linear Independence hypothesis. Then the natural densities of the sets \[S_k^+ = \left\{X \in \Z^+ \ |\ \ET_k(X) > 0 \right\} \quad \text{and} \quad S_k^- = \left\{X \in \Z^+ \ |\ \ET_k(X) < 0 \right\}\] exist and are given by \[\delta(S_k^+) = \delta(S_k^-) = \frac12.\]
\end{cor}

This result contrasts the distribution of the error term in Chebyshev's bias for prime races; under the assumption of the generalized Riemann hypothesis and the generalized Linear Independence hypothesis for $L$-functions in the classical case, and just the assumption of the generalized Linear Independence hypothesis in the function field case, the primes are biased towards quadratic non-residues \cite{RS94, Cha08}. Our result also contrasts the distribution of the error term in the function field analogue of P\'olya's conjecture \cite{Hum12}.

Continuing to assume the Linear Independence hypothesis, we find an explicit expression for the error bound \[B_k(C/\F_q) := \limsup_{X \to \infty} |\ET_k(X)|\] and study how it behaves on average over the class of function fields corresponding to the family of hyperelliptic curves in $\mathcal{H}_{2g + 1, q^n}$, where $\mathcal{H}_{2g + 1, q^n}$ denotes the set of curves given by the affine model $y^2 = f(x)$ for a polynomial $f \in \mathbb{F}_{q^n}[x]$ of degree $2g + 1$.  Inspired by the work of Cha \cite{Cha17} and Humphries \cite{Hum13}, we apply tools from random matrix theory, which require us to study the renormalized bound \[\widetilde{B}_k(C/\F_q) := \limsup_{X \to \infty} \frac{|\ET_k(X)|}{q^{g-g/k-1/2}}.\] We then deduce the following theorem on the size of the error term in the limit of large $q$. 
\begin{thm}\label{intro_thm:kfree_global}
    For any fixed $\beta > 0$, the quantity \[\lim_{n \to \infty} \frac{\#\{C \in \mathcal{H}_{2g + 1, q^n} \mid \widetilde{B}_k(C/\mathbb{F}_{q^n}) \leq \beta\}}{\#\mathcal{H}_{2g + 1, q^n}}\] is $0$ if $0 < \beta \leq 1$, and strictly between $0$ and $1$ if $\beta > 1$. 
\end{thm}

Theorem \ref{intro_thm:kfree_global} parallels Humphries' results that most hyperelliptic curves $C \in \mathcal{H}_{2g+1, q^n}$ do not satisfy the Mertens conjecture for $C/\F_{q^n}$ in the large $n$ limit, but a positive proportion of hyperelliptic curves satisfy the $\beta$-Mertens conjecture when $\beta > 1$ \cite{Hum13}. Meanwhile, for $k$-free numbers in the classical setting, Mossinghoff, Oliveira e Silva, and Trudigan \cite{MOT21} showed that \[\liminf_{x \to \infty} \frac{R_k^*(x)}{x^{1/2k}} < -0.74969 \quad \text{and} \quad \limsup_{x \to \infty} \frac{R_k^*(x)}{x^{1/2k}} > 0.74969\] for $k$ sufficiently large, and if the Linear Independence hypothesis in the classical setting is true, then the oscillations in $R_k^*(x)/x^{1/2k}$ are in fact unbounded. On the other hand, in the function field setting, the oscillations in $\ET_k(X)$ are bounded for any \emph{fixed} function field (if all zeros of the zeta function are simple), but Theorem \ref{intro_thm:kfree_global} implies that for any $\beta > 0$, a nontrivial proportion of hyperelliptic function fields in the large $n$ limit have $\widetilde{B}(C/\mathbb{F}_{q^n}) > \beta$. 

In the limit of large $q$, we additionally find that the maximal values of $\ET_k(X)$ are smaller for certain residue classes of $X$ modulo $k$. We are therefore motivated to study the global behavior of this bound across different residue classes modulo $k$. For each $0 \leq b \leq k-1$, we define \[B_k(C/\mathbb{F}_q, b+2g) := \limsup_{Y \to \infty} |\ET_k(b + 2g + kY)|.\] To apply the techniques we used in proving Theorem \ref{intro_thm:kfree_global}, we renormalize $B_k(C/\mathbb{F}_q, b+2g)$ to account for the aforementioned differences across residue classes modulo $k$ and define accordingly \begin{align*} \widetilde{B}_k(C/\mathbb{F}_q, b+2g) := \limsup_{Y \to \infty} \frac{|\ET_k(b + 2g + kY)|}{(b+1)q^{g - g/k - 1/2 - b/2k}}.\end{align*} Our results on the behavior of $\widetilde{B}_k(C/\mathbb{F}_q,b+2g)$ are summarized in the following theorem, which can be compared with Theorem \ref{intro_thm:kfree_global}. 

\begin{thm}\label{intro_thm:kfree_xmodk}
    For any $0 \leq b \leq k - 2$, the quantity \[\lim_{n \to \infty} \frac{\#\{C \in \mathcal{H}_{2g + 1, q^n} \mid \widetilde{B}_k(C/\mathbb{F}_{q^n}, b + 2g) \leq \beta\}}{\#\mathcal{H}_{2g + 1, q^n}}\] is $0$ if $0 < \beta \leq 1$, and strictly between $0$ and $1$ if $\beta > 1$. 
\end{thm}

Interestingly, the fluctuations in the size of $\widetilde{R}_k(X)$ do not appear to have a counterpart in the classical setting, and in Section \ref{subsec:global_kfree}, we discuss the origin for this phenomenon in the function field setting. In short, these fluctuations occur when summing over different powers of the roots of unity in (\ref{eqn:k_free_eq}).

The case $b = k - 1$ (which is not covered by Theorem \ref{intro_thm:kfree_xmodk}) is more difficult. When expressing $\widetilde{B}_k(C/\mathbb{F}_{q^n}, k - 1 + 2g)$ in terms of the zeros of the associated zeta function, the leading term is significantly more complicated, as the sum over the roots of unity causes the simpler terms (which appeared as the leading term for other values of $b$) to cancel out. This difficulty is discussed in more detail at the end of Subsection \ref{subsec:global_kfree}.

\subsection{Summatory Function of the Totient Function}

The same techniques can be used to study the function field analogues of the summatory function of the Euler totient function. In the classical setting, the Euler totient function is defined as
\[\phi(n) := n \prod_{\substack{p\mid n \\ p \text{ prime}}} \left(1 - \frac{1}{p}\right),\] 
and the associated summatory function is given by \[F^*_\phi(x) := \sum_{n \leq x} \phi(n).\]
Studying the limiting behavior of the classical error term $R_\phi^*(x)$ has been a standing goal of many authors --- in the classical setting, similar to the case of $k$-free numbers, one can show that 
\[F^*_\phi(x) = \frac{3}{\pi^2}x^2 + R_\phi^*(x).\] The best unconditional result in this direction is given by Walfisz \cite{Wal63}, who showed that \[R_\phi^*(x) = O(x(\log x)^{2/3}(\log \log x)^{4/3}).\] 
Unlike in the case of $k$-free numbers, however, the classical error term $R_\phi^*(x)$ is split into an \emph{analytic} error term, which is written in terms of zeros of the Riemann zeta function, and an \emph{arithmetic} error term \cite{KW10}, which is unexpected because it does not appear in the study of similar summatory functions, such as in the study of $k$-free numbers. In fact, Montgomery \cite{montgomery1987} showed that the overall error term $R_\phi^*(x)$ is quite large, with \[R_\phi^*(x) = \Omega_{\pm}(x\sqrt{\log\log x}),\] but Kaczorowski and Wiertelak \cite{KW10} proved that the arithmetic term is the main contributor to this lower bound, with \[R_\phi^{\mathrm{AR}}(x) = \Omega_{\pm}(x\sqrt{\log\log x}).\] On the other hand, under the assumption of the Riemann hypothesis, they showed that the analytic error term satisfies the more modest upper bound of \[R_\phi^{\mathrm{AN}}(x) = O_\varepsilon(x^{1/2 + \varepsilon}).\]

In the function field setting, however, we show that the arithmetic error term does not appear, i.e., the error term for the summatory totient function can be written solely as a sum across zeros of the associated zeta function.

We define the function field analogue of the totient function by 
\[
\Phi(D) := \mathcal{N}D \prod_{\substack{P\mid D \\ P \text{ prime}}}\left(1 - \frac{1}{\mathcal{N}P}\right)
\] for all effective divisors $D$ and its associated summatory function as
\[
F_{\Phi}(X) := \sum_{\deg(D) < X} \Phi(D).
\]
Using the same techniques as in the $k$-free divisor case (inspired by the work of Cha \cite{Cha17} and Humphries \cite{Hum12, Hum13}), we show that when the zeta function $Z_{C/\F_q}(u)$ has simple zeros at $u = \gamma_1^{-1}, \dots, \gamma_{2g}^{-1}$, the explicit expression for the normalized error term \[\widetilde{R}_\Phi(X) := \frac{F_\Phi(X) - \MT_\Phi(X)}{q^{X/2}}\] is given by 
\begin{align}
\ET_\Phi(X) = - \sum_{j=1}^{2g} \frac{Z_{C/\F_q}(\overline{\gamma_j})}{Z_{C/\F_q}'(\gamma_j^{-1})} \frac{\gamma_j}{\gamma_j - 1} e^{i \theta(\gamma_j)X} + O_{q,g}\left( \frac{1}{q^{X/2}} \right), \label{eqn:totient_eq}
\end{align}
where the main term is \[\MT_\Phi(X) =\frac{q^{1-g}h_{C/\F_q}}{\zeta_{C/\F_q}(2)(1-q^{-1})(q^2-1)}q^{2X}\] and $h_{C/\F_q}$ denotes the class number of $C/\mathbb{F}_q$. As in the case of $k$-free numbers (\ref{eqn:k_free_eq}), note that the normalized error term $\ET_k(X)$ is $O(1)$ as $X$ approaches infinity. Moreover, in contrast to the classical case, observe that an analogue of the arithmetic error term does \textit{not} appear in the expression for the function field error term. 

When $Z_{C/\F_q}(u)$ has zeros of multiple order, we proceed to show that \[R_\Phi(X) = O(X^{r-1}q^{X/2}) \quad \text{and} \quad \limsup_{X \to \infty} \frac{|R_\Phi(X)|}{X^{r - 1}q^{X/2}} > 0,\] where $r$ is the maximal order of a zero of $Z_{C/\F_q}(u)$. 

As in the case of $k$-free divisors, we use (\ref{eqn:totient_eq}) to show that the normalized error term $\ET_\Phi(X)$ possesses a limiting distribution. With this limiting distribution, we compute the natural density of the set \[\mathcal{S}_{\Phi}(\beta) := \{X \in \Z^+\ |\ |\ET_{\Phi}(X)| \leq \beta\}.\]

\begin{thm}\label{intro_thm:totient_lim_dist}
Let $C/\F_q$ be a function field with genus $g \geq 1$, and suppose that $C$ satisfies the Linear Independence hypothesis. For a real $\beta > 0$, the natural density \[\delta(\mathcal{S}_\Phi(\beta)) = \lim_{Y \to \infty} \frac{1}{Y} \# \{1 \leq X \leq Y \ |\ |\ET_\Phi(X)| \leq \beta \}\] exists and is given by
\[\delta(\mathcal{S}_\Phi(\beta)) = m\left(\sum_{j=1}^g 2\left|\frac{Z(\overline{\gamma_j})}{Z'(\gamma_j^{-1})}\frac{\gamma_j}{\gamma_j-1}\right|\cos(\theta_j) \in [-\beta,\beta]\right).\] 
\end{thm}

Note that the above result is an analogue to \cite[Theorem 1.6]{Hum12}. 

Similar to Corollary \ref{intro_cor:kfree_no_bias} in the setting of $k$-free divisors, we use the translation invariance of the Lebesgue measure to show that $\ET_\Phi(X)$ is unbiased. 

\begin{cor}\label{intro_cor:totient_no_bias}
    Let $C/\mathbb{F}_q$ be a function field with genus $g \geq 1$, and suppose that $C$ satisfies the Linear Independence hypothesis. The natural densities of the sets \[S_\Phi^+ = \left\{X\in \Z^+ \ |\ \ET_\Phi(X) > 0 \right\} \quad \text{and}\quad S_\Phi^- = \left\{X \in \Z^+ \ |\ \ET_\Phi(X) < 0 \right\}\] exist and are given by \[\delta(S_\Phi^+) = \delta(S_\Phi^-) = \frac12.\]
\end{cor}

As before, contrast this result to the analogous error terms in the study of Chebyshev's bias \cite{RS94, Cha08} and the P\'olya conjecture in function fields \cite{Hum12}. 

Finally, we analyze the average behavior of \[B_\Phi(C/\F_q) = \limsup_{X \to \infty} |\ET_\Phi(X)| \] over function fields corresponding to $\mathcal{H}_{2g + 1, q^n}$ in the limit of large $n$. As in the case of $k$-free numbers, we study the renormalized bound \[\widetilde{B}_{\Phi}(C/\F_q) = \limsup_{X \to \infty} \frac{|\ET_\Phi(X)|}{q^{2g-2}}\] and deduce the following theorem.

\begin{thm}\label{intro_thm:totient_global}
    For any fixed $\beta > 0$, the quantity \[\lim_{n \to \infty} \frac{\#\{C \in \mathcal{H}_{2g + 1, q^n} \mid \widetilde{B}_\Phi(C/\mathbb{F}_{q^n}) \leq \beta\}}{\# \mathcal{H}_{2g+1,q^n}}\] is $0$ if $0 < \beta \leq 1$, and strictly between $0$ and $1$ if $\beta > 1$. 
\end{thm}

The structure of the paper is as follows. In Section \ref{sec:kfree}, we study the distribution of $k$-free divisors in $C/\F_q$. First, the explicit expression for the error term (\ref{eqn:k_free_eq}) is computed in Subsection \ref{subsec:error_term_kfree}. Subsection \ref{subsec:multiplezeros_kfree} deals with the order of growth of the error term when the zeta function has a zero of multiple order. The limiting distribution for the error term is computed in Subsection \ref{subsec:limiting_distribution_kfree}. In Subsection \ref{subsec:global_kfree}, we analyze the global behavior of this error term on families of hyperelliptic curves in the large $q$ limit.
Analogous results are given for the totient function in Section \ref{sec:totient}.

\subsection*{Acknowledgements} The authors would like to thank Peter Humphries for supervising this project and Ken Ono for his generous support. The authors were participants in the 2022 UVA REU in Number Theory. They are grateful for the support of grants from the National Science Foundation
(DMS-2002265, DMS-2055118, DMS-2147273), the National Security Agency (H98230-22-1-0020), and the Templeton World Charity Foundation.

\section{Distribution of \texorpdfstring{$k$}{k}-Free Divisors in Function Fields}\label{sec:kfree}

In this section, we study the distribution of $k$-free divisors in a given function field $C/\F_q$. More precisely, recalling the $k$-free indicator function $\mu_k$, we examine the summatory function \[Q_k(X) := \sum_{\substack{D \geq 0\\ 0 \leq \deg(D) < X}} \mu_k(D),\] whose value at each $X$ gives the number of $k$-free effective divisors with degree less than $X$. In Section \ref{subsec:error_term_kfree}, we find an expression for $Q_k(X)$ in terms of a \emph{main term} $\mathrm{MT}_k(X)$ and an \emph{error term} $R_k(X)$. When the zeros of the associated zeta function $\zeta_{C/\F_q}(s)$ are simple, we find an explicit expression for $R_k(X)$ in terms of these zeros. In Section \ref{subsec:multiplezeros_kfree}, we show that \[R_k(X) = O(X^{r-1}q^{X/2k}), \quad \limsup_{X \to \infty} \frac{|R_k(X)|}{O(X^{r-1}q^{X/2k})} > 0\] whenever the maximal order of a zero of $\zeta_{C/\F_q}(s)$ is $r$. In Section \ref{subsec:limiting_distribution_kfree}, we show that the \textit{normalized} error term $\ET_k(X)$ has a limiting distribution whenever the zeros of $\zeta_{C/\F_q}(s)$ are simple. Under the stronger assumption of the Linear Independence hypothesis, we can explicitly determine the natural density of the subsets \[\mathcal{S}_{k}(\beta) := \{X \in \Z^+ \ |\ |\ET_k(X)| \leq \beta\}\] for any $\beta \in \R^+$. Finally, in Section \ref{subsec:global_kfree}, we first assume the Linear Independence hypothesis to derive an explicit bound on the {normalized} error term $\ET_k(X)$, then use techniques from random matrix theory to study the behavior of this bound on function fields of hyperelliptic curves over $\F_{q^n}$ in the limit of large $n$. 
\subsection{Computation of the Error Term}\label{subsec:error_term_kfree}

To find an expression for the summatory function $Q_k(X)$, we study the Dirichlet series \[D_k(s) := \sum_{D \geq 0} \frac{\mu_k(D)}{\mathcal{N}D^s},\] which converges absolutely when $\Re(s) > 1$, and compare its coefficients to an explicit expression in terms of the zeta function $\zeta_{C/\F_q}(s)$. 

\begin{lem}\label{lem:kfree_dirichlet_zeta}
For a function field $C/\mathbb{F}_q$ with associated zeta function $\zeta_{C/\F_q}(s) = Z_{C/\F_q}(u)$, the Dirichlet series associated with $\mu_k$ is given by \[D_k(s) = \frac{\zeta_{C/\F_q}(s)}{\zeta_{C/\F_q}(ks)} = \frac{Z_{C/\F_q}(u)}{Z_{C/\F_q}(u^k)}.\] 
\end{lem}

\begin{proof}
By the uniqueness of analytic continuation, it suffices to prove this equality when $\Re(s) > 1$. Since $\mu_k$ is multiplicative, we have
\begin{align*}
    \sum_{D \geq 0} \frac{\mu_k(D)}{\mathcal{N}D^s} &= \prod_{P \text{ prime}} \sum_{n=0}^{k-1} \frac{1}{\mathcal{N}P^{ns}} 
    = \prod_{P \text{ prime}} \left(\frac{1 - \mathcal{N}P^{-ks}}{1 - \mathcal{N}P^{-s}}\right) 
    = \frac{\zeta_{C/\F_q}(s)}{\zeta_{C/\F_q}(ks)}. \qedhere
\end{align*}
\end{proof}

With this explicit expression, we can now proceed to examine the summatory function $Q_k(X)$. We will first separately handle the case where $C/\F_q$ has genus $g = 0$, i.e., $C/\F_q = \F_q(T)$.  In this setting, we let $|f| := q^{\deg f}$ for any $f \in \F_q[T]$. If $f$ is a monic irreducible polynomial, then the localization of $\F_q[t]$ at the prime ideal $(f)$ is a finite prime whose norm coincides with $|f|$. By convention, we study a modified version of the function $Q_{k}(X)$. Here, we define \[\mu_k(f):= \begin{cases} 1& \text{if $f$ is not divisible by $m^k$ for any monic irreducible $m \in \F_q[T]$} \\ 0 &\text{if $f$ is divisible by $m^k$ for some monic irreducible $m \in \F_q[T]$} \end{cases}\] 
and seek to study the summatory function \[Q_{k,0}(X) := \sum_{\substack{f ~ \text{monic} \\ \deg(f) < X}} \mu_k(f),\] where we sum over monic polynomials in the polynomial ring $\F_q[T]$. 

Since the monic irreducible polynomials in $\F_q[T]$ correspond to sums of the finite primes in the divisor group of $\F_q(T)$, our modified definition simply ignores the divisors in $\F_q(T)$ divisible by the prime at infinity. We observe \[Q_{k}(X) = Q_{k,0}(X) + Q_{k,0}(X-1) + \cdots + Q_{k,0}(X-k+1),\] so finding an expression for $Q_{k,0}(X)$ is equivalent to finding an expression for $Q_{k}(X)$. Similarly, we work with the zeta function associated to the polynomial ring $\mathbb{F}_q[T]$ rather than the function field $\mathbb{F}_q(T)$. The zeta function of $\F_q[T]$ is given by \[\zeta_{\F_q[T]}(s) = \sum_{\substack{f \in \F_q[T] \\ f ~ \text{monic}}} \frac{1}{|f|^s}.\] For brevity, we will denote this zeta function by $\zeta_0(s)$.

The zeta functions of the ring $\F_q[T]$ and the field $\F_q(T)$ are in fact closely related, with $\zeta_0(s) = \zeta_{\F_q(T)}(s)(1 - q^{-s})$. Since $|f|$ is a power of $q$ for all $f$, the function $\zeta_0(s)$ can be rewritten as a function in the variable $u := q^{-s}$. We will denote this function by $Z_0(u) = \zeta_0(s)$. We have the following relationship between the Dirichlet series \[D_{k,0}(s) := \sum_{f ~\text{monic}} \frac{\mu_k(f)}{|f|^s}\] and the zeta function $\zeta_0(s)$. The proof of the following lemma is very similar to the proof of Lemma \ref{lem:kfree_dirichlet_zeta}.

\begin{lem} \label{lem:kfree_genus_zero_zeta}
When $\Re(s) > 1$, we have the equality \[\sum_{f ~ \mathrm{monic}} \frac{\mu_k(f)}{|f|^s} = \frac{\zeta_0(s)}{\zeta_0(ks)} = \frac{Z_0(u)}{Z_0(u^k)}.\]
\end{lem}

A closed form expression for $Q_{k,0}(X)$ is immediate. 

\begin{prop}\label{prop:genus_zero_kfree}
\[Q_{k,0}(X) = \frac{q^{X}-1}{q-1} - \frac{q(q^{X - k }-1)}{q-1} \cdot \mathbf{1}_{X > k}.\]
\end{prop}

\begin{proof} By Lemma \ref{lem:kfree_genus_zero_zeta} and the closed form expression $Z_0(u) = \frac{1}{1 - qu}$, we can rewrite $D_{k,0}(s)$ as 
\[
\sum_{f ~ \text{monic}} \frac{\mu_k(f)}{|f|^s} = \frac{1 - qu^k}{ 1- qu} \text{ and } \sum_{f ~ \text{monic}} \frac{\mu_k(f)}{|f|^s} = \sum_{n=0}^\infty a_nu^n,
\]
where $a_n = \sum_{\deg(f) = n} \mu_k(f)$. By
expanding the geometric series $\frac{1 - qu^k}{1 - qu}$ and comparing coefficients, we find that
\[
Q_{k,0}(X) = \sum_{n < X} a_n = \sum_{n = 0}^{X - 1} q^n - \left(\sum_{m = 1}^{X - k} q^m\right) \cdot \mathbf{1}_{X > k} = \frac{q^X - 1}{q - 1} - \frac{q(q^{X - k} - 1)}{q - 1} \cdot \mathbf{1}_{X > k}. \qedhere
\] 
\end{proof}

Notice that since $Z_0(u)$ has no zeros, the error term $R_k(X) = -\frac{1}{q - 1} + \frac{q}{q - 1} \cdot \mathbf{1}_{X > k}$ does not oscillate once $X > k$, in contrast to the genus $g \geq 1$ case. 

We proceed to study the original summatory function $Q_{k}(X)$ for an arbitrary function field $C/\F_q$ with genus $g \geq 1$. In particular, our next task is to write $Q_k(X)$ in terms of a main term $\mathrm{MT}_k(X)$ and an error term $R_k(X)$ when $\zeta_{C/\F_q}(s)$ has simple zeros. The strategy for finding this explicit formula is to apply Cauchy's residue theorem to a contour integral motivated by Lemma \ref{lem:kfree_dirichlet_zeta}, following the techniques of \cite{Cha17} and \cite{Hum12} in studying the summatory function of the M\"obius function over function fields. Let $\gamma_1,\ldots,\gamma_{2g}$ denote the inverse zeros of the zeta function $Z_{C/\F_q}(u)$, indexed so that $\gamma_{j+g}= \overline{\gamma_j}$ for $j = 1,\ldots,g$. For the sake of convenience, we will henceforth drop subscripts and denote the zeta function of $C/\F_q$ by $\zeta(s) = Z(u)$. We prove the following result on the limiting behavior of $Q_k(X)$ when $Z(u)$ has simple zeros.

\begin{prop}\label{prop:kfree_simple_zeros}
Suppose $C/\F_q$ has genus $g \geq 1$ and that the zeros of $Z(u)$ are all simple. For $j= 1,\ldots,2g$ and $\ell =0,\ldots, k-1$, we let $\gamma_{j,\ell} := q^{1/2k}e^{i(\theta(\gamma_j) +2\pi\ell)/k}$, where $\theta(\gamma_j)$ denotes the argument of $\gamma_j$ for $j = 1,\ldots,2g$. Then, as $X$ tends to infinity, 
\begin{align*}
\ET_k(X) &:= \frac{Q_k(X) - \MT_k(X)}{q^{X/2k}} \\[10pt] &=  - \sum_{j=1}^{2g} \sum_{\ell=0}^{k-1} \frac{Z(\gamma_{j, \ell}^{-1})}{k\gamma_{j, \ell}^{1-k}Z'(\gamma_j^{-1}) } \frac{\gamma_{j, \ell}}{\gamma_{j, \ell} - 1} e^{iX(\theta(\gamma_j) + 2\pi \ell)/k} + O_{q,g} \left({q^{-X/2k}}\right),
\end{align*}
where 
\[
\MT_k(X) := \frac{q^{1-g}h}{\zeta(k)(q-1)^2}q^X,
\] and $h$ is the class number of $C/\F_q$. 
\end{prop}

\begin{proof}
For $\rho > 1$, consider the contour $C_\rho = \{z \in \mathbb{C} : |z| = \rho\}$ oriented counterclockwise, and the corresponding contour integral 
\begin{align}\label{eqn:kfree_integral}
\frac{1}{2 \pi i} \oint_{C_\rho}\frac{1}{u^{N+1}}\frac{Z(u)}{Z(u^k)} \,du.
\end{align}
Rewrite $Z(u)/Z(u^k)$ as
\begin{align}
    \frac{Z(u)}{Z(u^k)} &= \frac{(1 - u^k)(1-qu^k)}{(1-u)(1-qu)}\frac{\prod_{j=1}^{2g} (1 - \gamma_j u )}{\prod_{j=1}^{2g} (1- \gamma_j u^k)}. \label{ID1_for_kfree}
\end{align} 
Note that the expression (\ref{ID1_for_kfree}) is valid whenever $u$ is not equal to $1$, $q^{-1}$, or $\gamma_{j, \ell}^{-1}$ for any $j = 1, \ldots, 2g$ and $\ell = 0, \ldots, k - 1$. We find an explicit expression for (\ref{eqn:kfree_integral}) by an application of Cauchy's residue theorem. The poles of the integrand of (\ref{eqn:kfree_integral}) inside $C_\rho$ occur at $u = 0$ with order $N+1$, at $u = q^{-1}$ with order 1, and at each $u = \gamma_{j,\ell}^{-1}$ with order 1.

Applying Cauchy's residue theorem along with  Lemma \ref{lem:kfree_dirichlet_zeta} and properties of $Z(u)$, we find that the integral (\ref{eqn:kfree_integral}) evaluates to
\[
    \sum_{\deg(D) = N} \mu_k(D) 
    + \sum_{j=1}^{2g}\sum_{\ell=0}^{k-1} \frac{Z(\gamma_{j, \ell}^{-1})}{k\gamma_{j, \ell}^{1-k} Z'(\gamma_j^{-1}) } \gamma_{j, \ell}^{N+1}
    - \frac{q^{N}}{\zeta(k)} \frac{q^{-g}h}{(1-q^{-1})}.
\]

Summing over all $0 \leq N \leq X - 1$ and evaluating the resulting geometric series yields
\begin{align}\label{eqn:qwerty}
    Q_k(X) 
    &= \frac{q^{1-g}h}{\zeta(k)(q-1)^2}q^X 
    - \sum_{j=1}^{2g}\sum_{\ell=0}^{k-1} \frac{Z(\gamma_{j, \ell}^{-1}) \gamma_{j, \ell}}{k\gamma_{j, \ell}^{1-k} Z'(\gamma_j^{-1}) }  \frac{\gamma_{j, \ell}^X}{\gamma_{j, \ell} - 1}
    + \varepsilon_{k}(X), 
\end{align}
where 
\[
\varepsilon_{k}(X) = -\frac{q^{1-g}h}{\zeta(k)(q-1)^2} + \sum_{j=1}^{2g} \sum_{\ell=0}^{k-1} \frac{Z(\gamma_{j, \ell}^{-1}) }{k\gamma_{j, \ell}^{1-k} Z'(\gamma_j^{-1}) } \frac{\gamma_{j, \ell}}{\gamma_{j, \ell}-1}
+ \frac{1}{2\pi i} \sum_{N=0}^{X-1}\oint_{C_\rho} \frac{1}{u^{N+1}}\frac{Z(u)}{Z(u^k)}\,du.
\]
Using (\ref{ID1_for_kfree}) to rewrite the integrand and the Riemann hypothesis for function fields, observe that
\begin{align}
    \left| \frac{1}{2\pi i} \oint_{C_\rho} \frac{1}{u^{N+1}} \frac{Z(u)}{Z(u^k)} \,du \right| &\leq \frac{1}{2 \pi} \oint_{C_\rho} \left|\frac{1}{u^{N+1}} \frac{Z(u)}{Z(u^k)}\right||du| \\[10pt] 
    &\leq \frac{(1 + \sqrt{q}\rho)^{2g} }{\rho^{N} (\rho-1)(q\rho-1)} \cdot \frac{(1+\rho^k)(1+q\rho^k)}{ (\sqrt{q}\rho^k - 1)^{2g}}.\label{eqn:kfree_inequality}
\end{align}
Note that we can make $\rho$ arbitrarily large without changing the validity of (\ref{eqn:qwerty}), since all singularities of the integrand of (\ref{eqn:kfree_integral}) are contained in a circle of radius $1$. A straightforward analysis of the right-hand side in (\ref{eqn:kfree_inequality}) shows that if $g > 1$, the integral (\ref{eqn:kfree_integral}) vanishes for all $N \geq 0$ as $\rho$ tends to infinity. On the other hand, if $g = 1$, the integral vanishes for all $N \geq 1$ as $\rho$ tends to infinity, which implies $\varepsilon_k(X)$ is a constant (not dependent on $X$), which can be calculated by setting $X = 1$ and using the fact that $Q_k(1) = 1$.
Dividing by $q^{X/2k}$ and using $\gamma_j = \sqrt{q}e^{i \theta(\gamma_j)}$ for each $j = 1,\ldots,2g$ gives the result.
\end{proof}

This concludes our analysis for the normalized error term $\ET_k(X)$ when $Z(u)$ has only simple zeros. In the case where $Z(u)$ has zeros of multiple order, we will show that $\ET_k(X)$, as we have defined it in this section, is not bounded as $X$ approaches infinity. More precisely, we show that \[R_k(X) := Q_k(X) - \MT_k(X)\] has order of growth $X^{r-1}q^{X/2k}$, where $r$ is the maximal order of a zero of $Z(u)$.

\subsection{Zeros of Multiple Order}\label{subsec:multiplezeros_kfree}

When all zeros of $Z(u)$ are simple, we found an explicit expression for $R_k(X)$ by computing
\[
\frac{1}{2\pi i} \oint_{C_\rho} \frac{D_k(u)}{u^{N+1}}\,du
\]
via Cauchy's residue theorem. Moving to the case of zeros of multiple order, the main term is unaffected since it comes from the residue at $u = q^{-1}$.  However, the expressions for the residues at $\gamma_{j, \ell}^{-1}$ are more complicated than in Proposition \ref{prop:kfree_simple_zeros}, since the poles are no longer assumed to be simple. To bound the growth of $R_{k}(X)$ when the maximal order of a zero of $Z(u)$ is $r$, the main difference is that we will instead need to calculate the residue at non-simple inverse zeros.

Additionally, we provide a lower bound on the growth of $R_k(X)$ by showing
\[
\limsup_{X \to \infty} \frac{|R_k(X)|}{X^{r-1}q^{X/2k}} > 0.
\]

The following lemma is needed for the first of these results.

\begin{lem}\label{lem:residue_calculation}
    Let $f$ be a function which is holomorphic in an open neighborhood of a fixed $\alpha^{-1}$. For a fixed positive integer $r$,
     \[\operatorname{Res}_{u = \alpha^{-1}}\left(\frac{f(u)}{(u - \alpha^{-1})^ru^{N + 1}}\right) = O(N^{r - 1}|\alpha|^N ).\] 
\end{lem}
\begin{proof}
        This residue is \[\operatorname{Res}_{u = \alpha^{-1}}\left(\frac{f(u)}{(u - \alpha^{-1})^ru^{N + 1}}\right) = \frac{1}{(r-1)!}\cdot \frac{d^{r-1}}{du^{r-1}}\bigg|_{u = \alpha^{-1}} f(u)u^{-N - 1}.\]
        
        Differentiate $r-1$ times using the product rule. Each term in the resulting sum is the product of a derivative of $f$ of some order, some polynomial in $N$ of degree at most $r - 1$, and $u^{-k}$ for some $N + 1 \leq k \leq N + r$. Note that the number of terms in the sum does not depend on $N$. Thus each term is $O(N^{r - 1}|\alpha|^N)$, and the entire sum is as well. 
    \end{proof}
    
    With this residue calculation, the upper bound follows easily.

\begin{prop}\label{prop:kfree_bound_above} 
    If all zeros of $Z(u)$ have order at most $r$, then
    \[|R_k(X)| = O(X^{r - 1}q^{X/2k}).\] 
\end{prop}

\begin{proof}
    If all zeros of $Z(u)$ have order at most $r$, then so do all zeros of $Z(u^k)$. If $\alpha$ is an inverse zero of $Z(u^k)$, then by Lemma \ref{lem:residue_calculation}, the residue at $\alpha^{-1}$ is $O(N^{r - 1}|\alpha|^N) = O(N^{r - 1}q^{N/2k})$, and the contribution to the sum is on the order of \[\sum_{N < X} N^{r - 1}q^{N/2k} \leq \sum_{N < X} X^{r - 1}q^{N/2k} = O(X^{r - 1}q^{X/2k}).\]  Summing over all $\alpha$ gives $|R_k(X)| = O(X^{r - 1}q^{X/2k})$. 
\end{proof}

We deal with the lower bound by following the general framework of \cite{Hum12}. We consider two cases depending on whether $\sqrt{q}$ is an inverse zero of $Z(u)$. The following lemmas will be helpful.

\begin{lem}[Vivanti--Pringsheim Theorem, {\cite[p. 235]{Rem91}}]\label{Vivanti--Pringsheim}
    Let $A(X)$ be a real-valued sequence, and suppose that there exists a positive integer $X_0$ such that $A(X)$ is of constant sign for all $X \geq X_0$. Furthermore, suppose that the supremum $v_c$ of the set of points $v \in [0, \infty)$ for which the sum 
    \[
    \sum_{X = X_0}^\infty A(X) v^{X-1}
    \]
    converges satisfies $v_c \leq 1$. Then the function
    \[
    H(u) = \sum_{X=1}^\infty A(X) u^{X-1}
    \]
    is holomorphic in the disk $|u| < v_c$ with a singularity at the point $v_c$.
\end{lem}

We also make use of the following identities.

\begin{lem}\label{cor2.9analogue}
    For $|u| < q^{-1}$, 
    \[
    \sum_{X=1}^\infty Q_k(X) u^{X-1} = \frac{Z(u)}{Z(u^k)(1-u)}.
    \]
\end{lem}
\begin{proof}
Consider the partial sum
\begin{align}\label{lem3.11 partial sum}
            \sum_{X = 1}^Y Q_k(X)u^{X - 1}(1 - u) &= \sum_{X = 0}^{Y - 1} (Q_k(X + 1) - Q_k(X))u^X - Q_k(Y)u^Y.
        \end{align}
        By definition, \[Q_k(X + 1) - Q_k(X) = \sum_{\deg D = X} \mu_k(D).\]  
        We have $Q_k(Y) = O(q^{Y})$ by the proof of Proposition \ref{prop:kfree_simple_zeros}, so $\lim_{Y \to \infty} |Q_k(Y)u^Y| = 0$. Thus, the series in (\ref{lem3.11 partial sum}) converges to \[\sum_{X = 1}^\infty Q_k(X)u^{X - 1}(1 - u) = D_k(u) = \frac{Z(u)}{Z(u^k)},\] as desired. 
\end{proof}

\begin{lem}[{\cite[Section 1.5]{Pet15}}]\label{combinatorial_lem}
    Let $|u| < 1$, and let $r$ be a positive integer. Then
    \[
    \sum_{X=1}^\infty X^{r-1}u^{X} = \frac{1}{(1-u)^r} \sum_{k=0}^{r-1} A(r-1,k) u^{k+1},
    \]
    where the coefficients 
    \[
    A(r-1,k) = \sum_{j=0}^k \binom{r}{j} (-1)^j (k+1-j)^{r-1}
    \]
    are the Eulerian numbers, satisfying the identity 
    $\sum_{k=1}^{r-1} A(r-1,k) =r!.$ 
\end{lem}

With these lemmas in place, we now consider the first case, where $\sqrt{q}$ is an inverse zero of $Z(u)$. Note that if $\sqrt{q}$ is an inverse zero of $Z(u)$, then it must have order $r \geq 2$ by the functional equation for $Z(u)$. 

\begin{prop}\label{prop:kfree_sqrtqcase}
    Let $g \geq 1$, and suppose that $\sqrt{q}$ is an inverse zero of $Z(u)$ with order $r \geq 2$. Then 
    \[
    \limsup_{X \to \infty} \frac{|R_k(X)|}{X^{r - 1}q^{X/2k}} > 0.
    \] 
\end{prop}
\begin{proof} Since ${R}_k(X) = O(X^{r-1}q^{X/2k})$, there exists a constant $c$ and a positive integer $X_0$ such that \begin{align}\label{eqn:O_inequality} |R_k(X)| \leq cX^{r-1}q^{X/2k}\end{align} for all $X \geq X_0$. First suppose that 
\begin{align}\label{kfree:signcondition}
    \frac{(-1)^rZ(q^{-1/2k})}{f^{(r)}(q^{-1/2k})(1-q^{-1/2k})} < 0,
\end{align} where $f(u) := Z(u^k)$. Consider \[\sum_{X = 1}^\infty (Q_k(X) - \mathrm{MT}_k(X) + cX^{r - 1}q^{X/2k})u^{X - 1}.\]  
    Set $d:= \frac{q^{1-g}h}{\zeta(k)(q-1)^2}$, so that
    \[
    \sum_{X=1}^\infty \mathrm{MT}_k(X) u^{X-1} = \frac{dq}{1-qu}
    \]
    for all $|u| < q^{-1}$.
    Combining with Lemmas \ref{cor2.9analogue} and \ref{combinatorial_lem}, we have 
    \begin{align}\label{kfree:powerseries_sqrtqcase}
         & \sum_{X = 1}^\infty (Q_k(X) - \mathrm{MT}_k(X) + cX^{r - 1}q^{X/2k})u^{X - 1} \\ \notag &= \frac{Z(u)}{Z(u^k)(1 - u)} - \frac{dq}{1 - qu} + \frac{c}{u(1 - q^{1/2k}u)^r}\sum_{\ell = 0}^{r - 1}A(r - 1, \ell)q^{(\ell + 1)/2k}u^{\ell + 1}.
    \end{align}
    Observe that the singularity at $q^{-1}$ on the right-hand side is removable. Thus, the right-hand side of (\ref{kfree:powerseries_sqrtqcase}) is holomorphic on $|u| < q^{-1/2k}$. By the  Vivanti--Pringsheim theorem, the left-hand side converges to a holomorphic function $G(u)$ for all $|u| < q^{-1/2k}$.

    Now multiply (\ref{kfree:powerseries_sqrtqcase}) by $(1 - q^{1/2k}u)^r$, and take the limit as $u$ approaches $q^{-1/2k}$ from the left through the real values. If the limit were negative, then the left-hand side of (\ref{kfree:powerseries_sqrtqcase}) would tend to negative infinity as $u$ approaches $q^{-1/2k}$ from the left. However, this is a contradiction since the series on the left-hand side of (\ref{kfree:powerseries_sqrtqcase}) has at most finitely many negative terms from $X = 0$ to $X_0 - 1$, and the sum from $X = X_0$ to infinity is nonnegative. Thus, the limit must be nonnegative.
    
    Using Lemma \ref{combinatorial_lem} and the observation that $1 - q^{1/2k}u = -q^{1/2k}(u - q^{-1/2k})$, we find that the right-hand side of (\ref{kfree:powerseries_sqrtqcase}) multiplied through by $(1 - q^{1/2k}u)^r$ tends to
    \[\frac{(-1)^rq^{r/2k}Z(q^{-1/2k})r!}{f^{(r)}(q^{-1/2k})(1 - q^{-1/2k})} + cq^{1/2k}r! \geq 0.\]
    Since (\ref{kfree:signcondition}) is negative, we get a positive lower bound on $c$. With $$c_0 := \frac{(-1)^{r}q^{r/2k}Z(q^{-1/2k})}{f^{(r)}(q^{-1/2k})(q^{1/2k} - 1)},$$ we have
    \[
    \liminf_{X\to \infty} \frac{R_k(X)}{X^{r-1}q^{X/2k}} \leq -c_0 < 0.
    \]
    
    On the other hand, if the left-hand side of (\ref{kfree:signcondition}) is positive, then an analogous argument proves
    \[
    \limsup_{X\to \infty} \frac{R_k(X)}{X^{r-1}q^{X/2k}} \geq \frac{(-1)^{r}q^{r/2k}Z(q^{-1/2k})}{f^{(r)}(q^{-1/2k})(q^{1/2k} - 1)} > 0.\qedhere
    \]
\end{proof}

If $\sqrt{q}$ is not an inverse zero of $Z(u)$ of maximal order, we obtain a stronger result through a very similar argument with slightly more complicated calculations.

\begin{prop}\label{prop:kfree_rmax}
    Let $C/\mathbb{F}_q$ be a function field with genus $g \geq 1$. Suppose that $\gamma$ is an inverse zero of $Z(u)$ of order $r \geq 2$, and that the order of the inverse zero at $\sqrt{q}$ is strictly less than $r$ (it may be 0). Then 
    \[
    \limsup_{X \to \infty} \frac{|R_k(X)|}{X^{r - 1}q^{X/2k}} > 0.
    \] 
\end{prop}

\begin{proof}
     Let $\alpha$ be any $k$th root of $\gamma$, and let $f(u) = Z(u^k)$, so that $\alpha$ is an inverse zero of $f$ with multiplicity $r$ as well. Moreover, let $\theta$ be the argument of $\alpha$. Since $R_k(X) = O(X^{r-1}q^{X/2k})$, there exists some $c \geq 0$ and a positive integer $X_0$ such that \[|R_k(X)| \leq cX^{r-1}q^{X/2k}\] for all $X \geq X_0$. Then, let $G(u)$ be the function to which (\ref{kfree:powerseries_sqrtqcase}) converges.  
     
     Consider the power series
     \begin{align}\label{kfree_big_powerseries}
     \sum_{X = 1}^\infty (Q_k(X) - MT_k(X) + cq^{X/2k}X^{r - 1})(1 - \cos(\phi - (X - 1)\theta))u^{X - 1}
     \end{align}
     where
     \[
        \phi = \pi - \arg\left( \frac{(-1)^r\alpha^r Z(\alpha^{-1})}{f^{(r)}(\alpha^{-1})(1 - \alpha^{-1})} \right).
     \]
    The reason for this choice of $\phi$ will become apparent later. Since
    \[
        \cos(\phi - (X - 1)\theta)u^{X - 1} = \frac{e^{i\phi}}{2}(e^{-i\theta}u)^{X - 1} + \frac{e^{-i\phi}}{2}(e^{i\theta}u)^{X - 1},
    \]
    for $|u| < q^{-1/2k}$, (\ref{kfree_big_powerseries}) is equal to \begin{align}\label{yuck}
    G(u) + \frac{e^{i\phi}}{2}G(ue^{-i\theta}) + \frac{e^{-i\phi}}{2}G(ue^{i\theta}).
    \end{align}
    We now multiply (\ref{yuck}) by $(1 - q^{1/2k}u)^r$ and analyze what happens as $u$ approaches $q^{-1/2k}$ from the left over real numbers. 
    
    Consider the limit of each term in (\ref{yuck}) individually. First observe that \[\lim_{\left(u \to q^{-1/2k}\right)^+} G(u)(1-q^{1/2k}u)^r = cq^{{1/2k}}r!\]
    
    Now consider $\frac{e^{i \phi}}{2} G(ue^{-i\theta})$. Let $v := u e^{-i\theta}$, so that $1 - q^{1/2k}u = -\alpha(v - \alpha^{-1})$. Then
    \begin{align} \label{kfree:definephi}
        \lim_{v \to \gamma^{-1}} \frac{Z(v)(-\alpha(v - \alpha^{-1}))^r}{f(v)(1-v)} = \frac{(-1)^{r}\alpha^r Z(\alpha^{-1})r!}{f^{(r)}(\alpha^{-1})(1 - \alpha^{-1})}.
    \end{align}
    Choosing $\phi$ to exactly cancel out the argument of (\ref{kfree:definephi}) and negate the resulting real number, 
    \[\lim_{u \to q^{-1/2k}} \frac{e^{i\phi}}{2}G(ue^{-i\theta})(1-q^{1/2k}u)^r =  -\frac{1}{2}\left|\frac{\alpha^r Z(\alpha^{-1})r!}{f^{(r)}(\alpha^{-1})(1 - \alpha^{-1})}\right|.\] Since $\overline{Z(u)} = Z(\overline{u})$ for all $u$, the third term in (\ref{yuck}) multiplied by $(1-q^{1/2k}u)^r$ tends to the same value in the limit. 
    
    Altogether, as $u$ approaches $q^{-1/2k}$ from the left through the real values, the product of $(1 - q^{1/2k}u)^r$ with (\ref{kfree_big_powerseries}) converges to
    \[
    cq^{1/2k} r! - \left|\frac{\alpha^r Z(\alpha^{-1})r!}{f^{(r)}(\alpha^{-1})(1 - \alpha^{-1})}\right|.
    \]
    By the same reasoning as in the proof of Proposition \ref{prop:kfree_sqrtqcase}, we deduce that the limit must be nonnegative, giving us the following lower bound on $c$:
    \[
    c \geq \left| \frac{\alpha^r Z(\alpha^{-1})}{q^{1/2k}f^{(r)}(\alpha^{-1})(1 - \alpha^{-1})}\right| > 0.
    \]

    An analogous argument shows
    \[
    \limsup_{X \to \infty} \frac{R_k(X)}{X^{r-1}q^{X/2}} \geq \left| \frac{\alpha^r Z(\alpha^{-1})}{q^{1/2k}f^{(r)}(\alpha^{-1})(1 - \alpha^{-1})}\right| > 0. \qedhere 
    \]
\end{proof}

\begin{rem}
Note that the Proposition \ref{prop:kfree_sqrtqcase} applies every time $\sqrt{q}$ is an inverse zero of $Z(u)$, but whenever $\sqrt{q}$ is not an inverse zero of maximal order, Proposition \ref{prop:kfree_rmax} is a stronger result. 
\end{rem}

\subsection{Limiting Distribution of the Error Term}\label{subsec:limiting_distribution_kfree}
In this subsection, we will compute the natural density in the positive integers of the set \[\mathcal{S}_k(\beta) = \{X \in \mathbb{Z}^+ \mid |\ET_k(X)| \leq \beta\}\] for a real $\beta > 0$ under the Linear Independence hypothesis. In particular, the natural density is given by
$$\delta(\mathcal{S}_k(\beta)) = \lim_{Y \to \infty} \frac{1}{Y} \# \{1 \leq X \leq Y \ |\ |\ET_k(X)| \leq \beta\}.$$

Recall that we let $\gamma_1,\ldots,\gamma_{2g}$ be the inverse zeros of $Z(u)$, ordered such that $\gamma_{j+g} = \overline{\gamma_j}$ for $j =1,\ldots, g$ and the argument of $\gamma_j$ lies in $[0,\pi]$ for $j= 1,\ldots,g$. Moreover, let $\gamma_{j,\ell} := q^{1/2k}e^{i(\theta(\gamma_j) +2\pi\ell)/k}$, where $\theta(\gamma_j)$ denotes the argument of $\gamma_j$ for $j = 1,\ldots,2g$. As seen in Section \ref{subsec:error_term_kfree}, we can write
\[
\ET_{k}(X) = E_{M,k}(X) + O_{q,g}(q^{-X/2k}),
\] 
where 
\[
E_{M,k}(X) := -\sum_{j=1}^{2g} \sum_{\ell=0}^{k-1} \frac{Z(\gamma_{j,\ell}^{-1})}{k\gamma_{j,\ell}^{1-k}Z'(\gamma_j^{-1})}\frac{\gamma_{j,\ell}}{\gamma_{j,\ell}-1}e^{iX(\theta(\gamma_j)+2\pi\ell)/k}.
\]  
In studying the limiting distribution of $\ET_k(X)$, we closely follow the methods employed by Humphries \cite{Hum12} while analyzing the Mertens conjecture in function fields, which are in turn modelled from the methods of Rubinstein and Sarnak \cite{RS94} in studying Chebyshev's bias. In particular, we will show that $\ET_k(X)$ has a limiting distribution $\nu_k$ as $X$ tends to infinity by first computing the limiting distribution of $E_{M, k}(X)$ using the Kronecker--Weyl theorem, and then using the Portmanteau theorem and Prohorov's theorem to relate it to the limiting distribution of $\ET_k(X)$. 
We also compute the Fourier transform of the probability measure $\nu_k$. The first few results will be stated under the condition that the inverse zeros of $Z(u)$ are simple. Later on, we will need to assume the stronger Linear Independence hypothesis.

 Consider the group homomorphism $\varphi: \Z \to \T^{kg}$ via 
 \begin{align}\label{eqn:uglyhom}
 \varphi(n) = (\exp(in(\theta(\gamma_j)+2\pi\ell)/k))_{\ell=0,j=1}^{k-1,g}.
 \end{align}
Let $G$ denote the topological closure of the subgroup $\varphi(\Z)$. Then by the Kronecker--Weyl theorem, $G$ is a closed subgroup of $\T^{kg}$. Our first step is to relate the limiting distribution of $E_{M, k}(n)$ to a limiting distribution on $G$  by constructing a pushforward measure. This is accomplished in the following proposition.

\begin{prop}\label{prop:kfree_main_error_distribution}
Let $C/\mathbb{F}_q$ be a function field of genus $g \geq 1$, and suppose that the zeros of $Z(u)$ are simple. Consider the function $\alpha_k: \T^{kg} \to \R$ defined by \[\alpha_k(z_{j,\ell})_{j=1,\ell=0}^{g,k-1} = -\sum_{j=1}^{g} \sum_{\ell=0}^{k-1} \left(\frac{Z(\gamma_{j,\ell}^{-1})}{k\gamma_{j,\ell}^{1-k}Z'(\gamma_j^{-1})}\frac{\gamma_{j,\ell}z_{j,\ell}}{\gamma_{j,\ell}-1} + \frac{Z\left(\overline{\gamma_{j,\ell}^{-1}}\right)}{k\overline{\gamma_{j,\ell}^{1-k}}Z'\left(\overline{\gamma_j^{-1}}\right)}\frac{\overline{\gamma_{j,\ell}z_{j,\ell}}}{\overline{\gamma_{j,\ell}}-1}\right)\] and the pushforward probability measure $\nu_k := (\alpha_k)_*\mu_G$, where $\mu_G$ is the Haar measure on $G$. For every continuous function $f: \R \to \R$, we have \[\lim_{N \to \infty} \frac{1}{N}\sum_{n=1}^N f(E_{M,k}(n)) = \int_\R f(x) \, d\nu_k(x),\] i.e., $\nu_k$ is the limiting distribution for $E_{M,k}(n)$. 
\end{prop}
\begin{proof} Let $f: \R \to \R$ be a continuous mapping, and let $\mu_G$ be the Haar measure on $G$. Observe that $\alpha_k \circ \varphi(n) = E_{M,k}(n)$ for each positive integer $n$. By the Kronecker--Weyl theorem \cite{Bai21}, 
\begin{align*}
    \lim_{k \to \infty} \frac{1}{k} \sum_{n=1}^k f(E_{M, k}(n)) & = \lim_{k \to \infty} \frac{1}{k} \sum_{n=1}^k [f \circ \alpha_k](\varphi(n)) 
     = \int_{\R} [f \circ \alpha_k](x) \, d\mu_G(x) 
     = \int_{\R} f(x) \, d\nu_k(x), 
\end{align*}
and the desired result follows.
\end{proof}

We now compute the limiting distribution of the overall normalized error term $\ET_k(X)$. In particular, we will show that it is equal to the limiting distribution $\nu_k = (\alpha_k)_* \mu_G$ of $E_{M, k}(X)$. 

\begin{prop}\label{prop:kfree_limiting_distribution}
Let $C/\mathbb{F}_q$ be a function field of genus $g \geq 1$, and suppose that the zeros of $Z(u)$ are simple. The overall error term $\ET_{k}(X)$ has limiting distribution $\nu_k = (\alpha_k)_*\mu_G$. That is, for every continuous function $f: \R \to \R$, we have \[\lim_{N \to \infty} \frac{1}{N}\sum_{X=1}^N f(\ET_k(X)) = \int_\R f(x) \, d\nu_k(x).\]
\end{prop}

\begin{proof}
For each positive integer $n$, let $\widetilde{\nu}_n$ denote the discrete uniform probability measure on the set $\{1,\ldots,n\}$. Then, define the pushforward measures $\lambda_n := (\ET_{k})_*\widetilde{\nu}_n$, so that for every continuous mapping $f: \R \to \R$ and each positive integer $n$, we have \[\frac{1}{n} \sum_{X=1}^n f(\ET_{k}(X)) = \int_{\{1,\ldots,n\}} [f\circ \ET_{k}](x) \, d\widetilde{\nu}_n(x) = \int_{\R} f(x)\, d\lambda_n(x).\] Since the function $\ET_{k}$ is bounded, every subsequence of the sequence $\{\widetilde{\nu}_n\}_{n=1}^\infty$ is tight. By Prohorov's theorem \cite[Theorem 5.1]{Bil99}, for each subsequence $\{\widetilde{\nu}_{n_\ell}\}_{\ell=1}^\infty$ of $\{\widetilde{\nu}_n\}_{n=1}^\infty$, there exists a subsubsequence $\{\widetilde{\nu}_{n_{\ell_t}}\}_{t=1}^\infty$ with weak limit $\xi$ for some probability measure $\xi$ on $\R$. We show that $ \xi = \nu_k$ for \textit{every} subsequence $\{\widetilde{\nu}_{n_\ell}\}_{\ell=1}^\infty$. By weak convergence and the Portmanteau theorem \cite[Theorem 2.1]{Bil99}, \begin{align}\label{eqn:kfree_subub_weak_convergence}\lim_{t \to \infty} \frac{1}{n_{\ell_t}} \sum_{X=1}^{n_{\ell_t}} f(\ET_{k}(u)) = \int_{\R} f(x) \, d\xi(x)\end{align} for every bounded Lipschitz continuous function $f: \mathbb{R} \rightarrow \mathbb{R}$. Thus, there exists a constant $c_f \geq 0$ such that the following inequalities hold: \begin{align}\label{eqn:lipschitz_ineq1}\frac{1}{n_{\ell_t}}\sum_{X=1}^{n_{\ell_t}} f(\ET_k(X)) \geq \frac{1}{n_{\ell_t}}\sum_{X=1}^{n_{\ell_t}} f(E_{M,k}(X)) - \frac{c_f}{n_{\ell_t}}\sum_{X=1}^{n_{\ell_t}} |O_{q,g}(q^{-X/2})| \\[10pt]\frac{1}{n_{\ell_t}}\sum_{X=1}^{n_{\ell_t}} f(\ET_k(X)) \leq \frac{1}{n_{\ell_t}}\sum_{X=1}^{n_{\ell_t}} f(E_{M,k}(X)) + \frac{c_f}{n_{\ell_t}}\sum_{X=1}^{n_{\ell_t}} |O_{q,g}(q^{-X/2})|\label{eqn:lipschitz_ineq2}.\end{align} 
By Proposition \ref{prop:kfree_main_error_distribution}, 
\[\lim_{t \to \infty} \frac{1}{n_{\ell_t}}\sum_{X=1}^{n_{\ell_t}} f(E_{M,k}(X)) = \int_{\R} f(x) \, d\nu_k(x),\] 
where $\nu_k = (\alpha_k)_*\mu_G$ is defined as in the statement of Proposition \ref{prop:kfree_main_error_distribution}. Notice that \[
\lim_{t \to \infty} \frac{c_f}{n_{\ell_t}}\sum_{X=1}^{n_{\ell_t}} |O_{q,g}(q^{-X/2})| = 0.\] Applying (\ref{eqn:kfree_subub_weak_convergence}) and taking the limit as $t\to\infty$ on both sides of (\ref{eqn:lipschitz_ineq1}) and (\ref{eqn:lipschitz_ineq2}), we have \[\lim_{t \to \infty} \frac{1}{n_{\ell_t}}\sum_{X=1}^{n_{\ell_t}} f(\ET_k(X)) = \int_{\R} f(x) \, d\xi(x) = \int_{\R} f(x) \, d\nu_k(x)\] for every Lipschitz continuous $f: \R \to \R$. Again by the Portmanteau theorem, we deduce that $\widetilde{\nu}_{n_{\ell_t}}$ converges weakly to $\nu_k$. Since this holds for every subsequence of $\widetilde{\nu}_n$, we conclude that $\nu_k$ is the weak limit of $\{\widetilde{\nu}_n\}_{n=1}^\infty$, and the desired result follows.
\end{proof}

Now that we have shown that $\nu_k$ is the limiting distribution of $\ET_k(X)$, we must find an explicit expression for $\nu_k$ to prove Theorem \ref{intro_thm:kfree_limdist}. The next lemma facilitates this.

\begin{lem}\label{lem:kfree_torus}
Assume that $C/\F_q$ satisfies the Linear Independence hypothesis. Let $G$ be the topological closure of $\varphi(\Z)$ in the $kg$-torus $\T^{kg}$, with $\varphi$ as defined in (\ref{eqn:uglyhom}). Then, $G$ is the union of $k$ disjoint $g$-tori $\T^g$.
\end{lem}

\begin{proof}
Fix a set of angles $\Theta := (\theta_1,\ldots,\theta_g)$ and fix some $t \in \{0,1,\ldots,k-1\}$. Thanks to the linear independence of $\{\theta(\gamma_1),\ldots,\theta(\gamma_j),\pi\}$ and the Kronecker--Weyl theorem, there exists an integer sequence $(X_s^{(\Theta)})_{s=1}^\infty$ such that \[\lim_{s \to \infty} \left(\exp\left(iX_s^{(\Theta)}\theta(\gamma_j)\right)\right)_{j=1}^g = (\exp(i\theta_j) - t\theta(\gamma_j)/k)_{j=1}^g.\] Thus, for every $\ell = 0,1,\ldots,k-1$, we have \begin{align*}\lim_{s \to \infty} \left(\exp\left(i\left(kX_s^{(\Theta)}+t\right)\left(\frac{\theta(\gamma_j) + 2\pi\ell}{k}\right)\right)\right)_{j=1}^g =\left(\exp\left(i(\theta_j + 2\pi [\ell t]/k)\right)\right)_{j=1}^g,\end{align*} where $[\ell t]  \equiv \ell t\pmod{k}$ with $[\ell t] \in \{0,1,\ldots,k-1\}$. Thus, we deduce that \begin{align}\label{eqn:lem2.15}\left(\left(\exp\left(i(\theta_j + 2\pi [\ell t]/k)\right)\right)_{j=1}^g\right)_{\ell=0}^{k-1}\end{align} belongs to $G$ for any $t \in \{0,1,\ldots,k-1\}$.  On the other hand, suppose that there exists an integer sequence $(X_s)_{s=1}^\infty$ such that \[\exp(iX_s(\theta(\gamma_j) + 2\pi \ell)/k)\] converges for all $j = 1,\ldots,g$ and $\ell = 0,1,\ldots,k-1$. By passing to a subsequence, we can assume that there exists $t \in \{0,1,\ldots,k-1\}$ such that $X_s \equiv t\pmod{k}$ for all $s$. Then, setting \[\theta_j := \lim_{s\to\infty}\exp(iX_s(\theta(\gamma_j)/k))\] for all $j = 1,\ldots,g$, it follows that \[\exp(iX_s(\theta(\gamma_j) + 2\pi \ell)/k) \to \exp(i(\theta_j + 2\pi [\ell t]/k)).\] In other words, every element of $G$ must take the form of (\ref{eqn:lem2.15}). Thus, we deduce that \[G = \bigcup_{t=0}^{k-1} \left\{\left(\left(\exp\left(i(\theta_j + 2\pi [\ell t]/k)\right)\right)_{j=1}^g\right)_{\ell=0}^{k-1}\ \bigg|\ (\theta_1,\ldots,\theta_g) \in [0,2\pi)^g\right\}.\] We claim that this union is disjoint. Suppose there existed $(\theta_1,\ldots,\theta_g)$ and $(\theta_1',\ldots,\theta_g')$ such that \[\left(\left(\exp\left(i(\theta_j + 2\pi [\ell t]/k)\right)\right)_{j=1}^g\right)_{\ell=0}^{k-1} = \left(\left(\exp\left(i(\theta_j' + 2\pi [\ell t']/k)\right)\right)_{j=1}^g\right)_{\ell=0}^{k-1}\] for some $t, t'$.  
This implies that \[\theta_j + 2\pi[\ell t]/k \equiv \theta'_j + 2\pi[\ell t']/k \pmod{2\pi}\] for every $j$ and every $\ell$. With $\ell = 0$, we see that $\theta_j' = \theta_j$ for each $j$ (as $\theta_j,\theta_j'$ are restricted to $[0,2\pi)$, we have $\theta_j = \theta_j'$). But then, for each $\ell$, it follows that \[2\pi[\ell t]/k \equiv 2\pi[\ell t']/k \pmod{2\pi} \implies \ell t \equiv \ell t' \pmod{k}.\] By taking $\ell$ to be relatively prime to $k$, we conclude the desired $t = t'$. For each fixed $t$, \[\left\{\left(\left(\exp\left(i(\theta_j + 2\pi [\ell t]/k)\right)\right)_{j=1}^g\right)_{\ell=0}^{k-1}\ \bigg|\ (\theta_1,\ldots,\theta_g) \in [0,2\pi)^g\right\}\] is a copy of $\T^g$, so we deduce that $G$ is the disjoint union of $k$ copies of $\T^g$. 
\end{proof}

Since the preceding lemma tells us that the subgroup $G$ is precisely the disjoint union of $k$ copies of the $g$-torus $\T^g$, the Haar measure $\mu_G$ is simply $1/k$ of the Lebesgue measure on each copy of $\T^g$. Denote the coordinates on the $t$-th copy of $\T^g$ by $\theta_{t,1},\ldots,\theta_{t,g}$. Henceforth, we will identify each copy of $\T^g$ with the direct product of $g$ copies of the interval $[0,2\pi)$. Under this identification, the function $\alpha_k$ from Proposition \ref{prop:kfree_main_error_distribution} is given by \begin{align*}\alpha^{(t)}_k(\theta_{t,j})_{j=1}^{g} = -\sum_{j=1}^g  2\Re\left(\sum_{\ell=0}^{k-1}\frac{Z(\gamma_{j,\ell}^{-1})}{k\gamma_{j,\ell}^{1-k}Z'(\gamma_j^{-1})}\frac{\gamma_{j,\ell}}{\gamma_{j,\ell}-1}e^{i(\theta_{t,j} + 2\pi[\ell t]/k)}\right)\end{align*} on the $t$-th copy of the torus. 

We are now ready to state and prove the main result of this section, computing the natural density of $\mathcal{S}_k(\beta)$ in the set of positive integers. Recall the probability measures $\lambda_n := (\ET_k)_\ast \widetilde{\nu}_n$ from the proof of Proposition \ref{prop:kfree_limiting_distribution}. Then for any Borel set $B$ and integer $n > 0$, 
\[
\lambda_n(B) = \widetilde{\nu}_n(\ET_k^{-1}(B)) = \frac{\#\left\{1 \leq r \leq n \mid \ET_{k}(r) \in B\right\}}{n}.
\] 
Observe that if the limit $\lim_{n \to \infty} \lambda_n([-\beta, \beta])$ exists, then it is precisely the natural density of $\mathcal{S}_k(\beta)$. Thus, to compute the natural density of $\mathcal{S}_k(\beta)$, it suffices to show that this limit exists and relate its value to the measure of subsets in $G$. 
For convenience, given a Borel set $B \subset \R$ and Borel-measurable function $f: [0, 2 \pi)^g \rightarrow \mathbb{R}$, we use the notation $m(f(\theta_1, \hdots, \theta_g) \in B)$ for $$m(\{(\theta_1, \hdots, \theta_g) \in [0, 2 \pi)^g ~ | ~ f(\theta_1, \hdots, \theta_g) \in B \})$$ where $m$ denotes the Lebesgue measure on $[0,2\pi)^g$.

{
\renewcommand{\thethm}{\ref{intro_thm:kfree_limdist}}
\begin{thm}
Let $C/\mathbb{F}_q$ be a function field of genus $g \geq 1$, and suppose that $C/\F_q$ satisfies the Linear Independence hypothesis. The natural density of the set $\mathcal{S}_k(\beta)$, denoted $\delta(\mathcal{S}_k(\beta))$, exists and \[\delta(\mathcal{S}_k(\beta)) = \sum_{t=0}^{k-1} \frac{1}{k} \cdot m\left(\ \sum_{j=1}^g 2\left|\sigma_{t,j}^{(k)}\right|\cos(\theta_{t,j}) \in [-\beta,\beta]\right),\] where \[\sigma_{t,j}^{(k)} := \sum_{\ell=0}^{k-1}\frac{Z(\gamma_{j,\ell}^{-1})}{k\gamma_{j,\ell}^{1-k}Z'(\gamma_j^{-1})}\frac{\gamma_{j,\ell}}{\gamma_{j,\ell}-1}e^{2\pi i[\ell t]/k}\] for each $j = 1,\ldots,g$ and $t =0,1,\ldots,k-1$.
\end{thm}
\addtocounter{thm}{-1}
}

\begin{proof}
By the translation invariance of the Lebesgue measure, we have the following equalities for any Borel set $B \subset \R$:

\begin{align}
    \notag\nu_k(B) &= \sum_{t = 0}^{k - 1} \frac{1}{k} \cdot m( \alpha_k^{(t)}(\theta_{t,1},\ldots,\theta_{t,g}) \in B)  \\[6pt] \notag &= 
    \sum_{t=0}^{k - 1} \frac{1}{k} \cdot m\left( -\sum_{j=1}^g  2\Re\left(\sum_{\ell=0}^{k-1}\frac{Z(\gamma_{j,\ell}^{-1})}{k\gamma_{j,\ell}^{1-k}Z'(\gamma_j^{-1})}\frac{\gamma_{j,\ell}}{\gamma_{j,\ell}-1}e^{i(\theta_{t,j}^* + 2\pi[\ell t]/k)}\right) \in B \right) \\[8pt]  &=
    \sum_{t = 0}^{k - 1} \frac{1}{k} \cdot  m\left( \sum_{j=1}^g 2|\sigma_{t, j}^{(k)}|\cos(\theta_{t,j}) \in B\right)\label{eqn:translation_invariance},
\end{align} 

where for each $j = 1,\ldots, g$ and each $t = 0,\ldots,k-1$, we let $\theta_{t,j}^* := \theta_{t,j} - (\arg(\sigma_{t, j}^{(k)})-\pi).$ The function $\sum_{j=1}^g 2|\sigma_{t, j}^{(k)}|\cos(\theta_j)$ is real analytic on $[0, 2 \pi)^g$ and is not uniformly constant, so its level sets have Lebesgue measure zero. In particular, $\nu_k(\{-\beta\} \cup \{\beta\}) = 0$, so by the Portmanteau theorem \cite[Theorem 2.1]{Bil99} and the fact that the measures $\lambda_n$ from the proof of Proposition \ref{prop:kfree_limiting_distribution} converge weakly to $\nu_k$, we have
\[
\delta(\mathcal{S}_k(\beta)) = \lim_{n \to \infty} \lambda_n([-\beta,\beta]) = \nu_k([-\beta,\beta]),
\] whence the desired result follows with the choice $B = [-\beta,\beta]$ in (\ref{eqn:translation_invariance}).
\end{proof}

A similar argument to the one we made in the above proof allows us to deduce that the normalized error term $\ET_k(X)$ is unbiased, i.e., it is positive and negative equally often.

{
\renewcommand{\thethm}{\ref{intro_cor:kfree_no_bias}}

\begin{cor}
    Let $C/\mathbb{F}_q$ be a function field of genus $g \geq 1$, and suppose that $C$ satisfies the Linear Independence hypothesis. Then the natural densities of the sets \[S_k^+ = \left\{X\in \Z^+ \ |\ \ET_k(X) > 0 \right\} \quad \text{and} \quad S_k^- = \left\{X \in \Z^+\ |\ \ET_k(X) < 0 \right\}\] exist and are given by \[\delta(S_k^+) = \delta(S_k^-) = \frac12.\]
\end{cor}
\addtocounter{thm}{-1}
}

\begin{proof}
From the proof of Theorem \ref{intro_thm:kfree_limdist}, we know
\begin{align*}
    \delta(\{X \in \Z^+\ |\ \ET_k(X) \in B\}) = \nu_k(B) = \sum_{t = 0}^{k - 1} \frac{1}{k} \cdot  m\left(\ \sum_{j=1}^g 2|\sigma_{t, j}^{(k)}|\cos(\theta_j) \in B\right)
\end{align*} for any Borel subset $B \subset \R$ with boundary measure zero (in particular, the sets $B = (-\infty,0)$ and $(0,\infty)$). By the translation and reflection invariance of the Lebesgue measure, it follows that \begin{align*}
    \delta(S_k^-) &= \sum_{t = 0}^{k - 1} \frac{1}{k} \cdot  m\left(\ \sum_{j=1}^g 2|\sigma_{t, j}^{(k)}|\cos(\pi - \theta_j) \in (-\infty,0)\right) \\[10pt] &= \sum_{t = 0}^{k - 1} \frac{1}{k} \cdot  m\left(\ -\sum_{j=1}^g 2|\sigma_{t, j}^{(k)}|\cos(\theta_j) \in (-\infty,0)\right) \\[10pt] &= \sum_{t = 0}^{k - 1} \frac{1}{k} \cdot  m\left(\ \sum_{j=1}^g 2|\sigma_{t, j}^{(k)}|\cos(\theta_j) \in (0,\infty)\right) = \delta(S_k^+),
\end{align*} as desired.
\end{proof}

To complete our analysis of the error term, we compute the Fourier transform of the limiting distribution $\nu_k$. In fact, the explicit expression for the Fourier transform will give us a different proof of Corollary \ref{intro_cor:kfree_no_bias} when $g \geq 3$. 

\begin{prop} \label{prop:fourier_kfree}
Assume that $C/\F_q$ satisfies the Linear Independence hypothesis. The Fourier transform of $\nu_k$ exists and is given by \[
\widehat{\mu}_k(y) = \frac{1}{k} \sum_{t=0}^{k-1} \prod_{j=1}^g J_0\left(2\left|\sum_{\ell=0}^{k-1}\frac{Z(\gamma_{j,\ell}^{-1})}{k\gamma_{j,\ell}^{1-k}Z'(\gamma_j^{-1})}\frac{\gamma_{j,\ell}}{\gamma_{j,\ell}-1} e^{2\pi i[\ell t]/k}\right|y\right),
\] 
where 
\[
J_0(z) = \sum_{m=0}^\infty \frac{(-1)^m(z/2)^{2m}}{(m!)^2}
\]
is the Bessel function of the first kind.
\end{prop}

\begin{proof}
The proof of this proposition is a routine calculation. See, for example, the proof of \cite[Theorem 3.4]{Cha08}, which is in turn inspired by the analogous proof in \cite{RS94}.
\end{proof}

\subsection{Global Behavior of the Error Term on Families of Hyperelliptic Curves}\label{subsec:global_kfree}

In this subsection, we will analyze how the error bound \[B_k(C/\mathbb{F}_q) := \limsup_{X \to \infty} |\ET(X)|\] behaves on average over a certain class of function fields, specifically the function fields corresponding to a family of hyperelliptic curves. First, we find an explicit expression for $B_k(C/\mathbb{F}_q)$ under the assumption that $C/\mathbb{F}_q$ satisfies the Linear Independence hypothesis.

For $a = 0,\ldots,k-1$ and $j = 1,\ldots,2g$, define \[c_{j, a} := \frac{1}{k}\sum_{\ell = 0}^{k - 1} \frac{Z(\gamma_{j, \ell}^{-1})}{\gamma_{j, \ell} - 1}e^{2\pi i \ell a/k},\] so that \[\ET_k(X) = -\sum_{j = 1}^{2g} \frac{c_{j, a}\gamma_j}{Z'(\gamma_j^{-1})}e^{iX\theta(\gamma_j)/k} + O_{q,g}(q^{-X/2k})\] whenever $X \equiv a \pmod{k}$.

\begin{lem}\label{lem:kfree_modk_dense}
    Suppose that $C/\F_q$ satisfies the Linear Independence hypothesis. Moreover, fix any $a= 0,\ldots,k-1$. The set \[\left\{\left(e^{iX\theta(\gamma_1)/k}, \ldots, e^{iX\theta(\gamma_g)/k}\right)\right\}_{X \equiv a \pmod{k}}\] is dense in the $g$-torus $\T^g$.
\end{lem}

\begin{proof}
    Fix any set of angles $(\theta_1,\ldots,\theta_g)$. Under the assumption of the Linear Independence hypothesis, the Kronecker--Weyl theorem tells us that there exists a sequence of integers $(Y_\ell)_{\ell=1}^\infty$ such that  \[\left(e^{iY_\ell\theta(\gamma_1)}, \ldots, e^{iY_\ell\theta(\gamma_g)}\right) \to \left(e^{i\theta_1 - ia\theta(\gamma_1)/k}, \ldots, e^{i\theta_g - ia\theta(\gamma_1)/k}\right). \] Then, for each positive integer $\ell$, set $X_\ell := Y_\ell k + a$, so that $X_\ell \equiv a \pmod{k}$ and \[\left(e^{iX_\ell\theta(\gamma_1)/k}, \ldots, e^{iX_\ell\theta(\gamma_g)/k}\right) \to \left(e^{i\theta_1}, \ldots, e^{i\theta_g }\right),\] proving the claim.
\end{proof}

With this lemma, we can deduce the following expression for $B_k(C/\F_q)$ by picking integers $X$ so that the terms $e^{iX\theta(\gamma_j)/k}$ nearly cancel the arguments of the coefficients $c_{j,a}\gamma_j/Z'(\gamma_j^{-1})$, in the style of the proof in Humphries \cite[Theorem 2.6]{Hum12}. 

\begin{prop}\label{prop:kfree_bcfq}
    Suppose that $C/\F_q$ satisfies the Linear Independence hypothesis. Then, \[B_k(C/\mathbb{F}_q) = \max_{0 \leq a \leq k - 1} \sum_{j = 1}^{2g} \left|\frac{c_{j, a}\gamma_j}{Z'(\gamma_j^{-1})}\right|.\] 
\end{prop}

\begin{proof}
    Since the inverse zeros of the zeta function come in conjugate pairs, note that \[-\sum_{j=1}^{2g} \frac{c_{j,a}\gamma_j}{Z'(\gamma_j^{-1})}e^{iX\theta(\gamma_j)/k} = -2\sum_{j=1}^{g}\Re\left(\frac{c_{j,a}\gamma_j}{Z'(\gamma_j^{-1})}e^{iX\theta(\gamma_j)/k}\right).\] For each $a = 0,\ldots,k-1$, there exists a sequence $X_{\ell}^{(a)}$ satisfying $X_{\ell}^{(a)} \equiv a \pmod{k}$ for each $\ell$ (by Lemma \ref{lem:kfree_modk_dense}) such that \[-2\sum_{j=1}^{g}\Re\left(\frac{c_{j,a}\gamma_j}{Z'(\gamma_j^{-1})}e^{iX_\ell^{(a)}\theta(\gamma_j)/k}\right) \to 2\sum_{j=1}^g \left|\frac{c_{j,a}\gamma_j}{Z'(\gamma_j)^{-1}}\right|,\] so it follows that \[\limsup_{X \to \infty} \ET_k(X) = \limsup_{X \to \infty}\left(-2\sum_{j=1}^{g}\Re\left(\frac{c_{j,a}\gamma_j}{Z'(\gamma_j^{-1})}e^{iX\theta(\gamma_j)/k}\right)\right) \geq \max_{0 \leq a\leq k-1}\sum_{j=1}^{2g}\left|\frac{c_{j,a}\gamma_j}{Z'(\gamma_j)^{-1}}\right|.\] The other direction of the inequality follows from the triangle inequality. Similarly, \[\liminf_{X \to \infty} \ET_k(X) = -\left(\max_{0 \leq a\leq k-1}\sum_{j=1}^{2g}\left|\frac{c_{j,a}\gamma_j}{Z'(\gamma_j)^{-1}}\right|\right),\] concluding the proof.
\end{proof}

We now relate the function field $C/\mathbb{F}_q$ to a unitary symplectic matrix, and use arguments from random matrix theory. Define $\mathrm{USp}_{2g}(\mathbb{C})$ to be the space of $2g \times 2g$ unitary symplectic matrices, i.e., matrices which satisfy the conditions 
\begin{center}
    $U^\dagger U = I_{2g}$\quad and \quad $U^\top JU = J$, where $J = \begin{pmatrix} 0 & I_g \\ -I_g & 0 \end{pmatrix}$,
\end{center}
where $U^\dagger$ is the conjugate transpose of $U$ and $U^\top$ is the transpose of $U$. 

 The eigenvalues of any $U \in \mathrm{USp}_{2g}(\mathbb{C})$ have absolute value $1$ and occur in complex conjugate pairs; conversely, for any angles $\theta_1$, \ldots, $\theta_{2g}$ with $\theta_{j + g} = -\theta_j$ for all $j = 1, \ldots, g$, the diagonal matrix with diagonal entries $e^{i\theta_1}$, \ldots, $e^{i\theta_{2g}}$ is in $\mathrm{USp}_{2g}(\mathbb{C})$. 

\begin{defn}
    Fix a function field $C/\mathbb{F}_q$, and let $\gamma_1$, \ldots, $\gamma_{2g}$ be the inverse zeros of $Z(u)$, indexed so that $0 \leq \theta(\gamma_j) \leq \pi$ and $\gamma_{j + g} = \overline{\gamma_j}$ for all $j = 1, \ldots, g$.
    Define its \emph{unitarized Frobenius conjugacy class} $\vartheta(C/\mathbb{F}_q)$ to be the conjugacy class of $\mathrm{USp}_{2g}(\mathbb{C})$ containing the diagonal matrix with diagonal $e^{i\theta(\gamma_1)}$, $e^{i\theta(\gamma_2)}$, \ldots, $e^{i\theta(\gamma_{2g})}$. 
\end{defn}

\begin{defn}
    For a matrix $U \in \mathrm{USp}_{2g}(\mathbb{C})$ with eigenvalues $e^{i\theta_1}$, \ldots, $e^{i\theta_{2g}}$, define its \emph{characteristic polynomial} as \[\mathcal{Z}_U(\theta) = \prod_{i = 1}^{2g}(1 - e^{i(\theta_j - \theta)}).\] 
\end{defn}

In particular, if the function field $C/\mathbb{F}_q$ has zeta function $Z(u) = \frac{L(u)}{(1 - u)(1 - qu)}$, the characteristic polynomial of $\vartheta(C/\mathbb{F}_q)$ can be written in terms of the polynomial $L$, as \[\mathcal{Z}_{\vartheta(C/\mathbb{F}_q)}(\theta) = L(q^{-1/2}e^{-i\theta}).\] 

After fixing an odd prime power $q$ and a genus $g$, define $\mathcal{H}_{2g + 1, q^n}$ to be the set of hyperelliptic curves defined by the affine model $y^2 = f(x)$ for a polynomial $f \in \mathbb{F}_{q^n}[x]$ of degree $2g + 1$. We will study the distribution of $B_k(C/\mathbb{F}_{q^n})$ for curves $C \in \mathcal{H}_{2g + 1, q^n}$, in the limit as $n$ tends to infinity. The reason to work with such curves is that the two following facts about them are known.

\begin{prop}[Chavdarov \cite{Cha97}, Kowalski \cite{Kowalski08}, see {\cite[Theorem 3.1]{Cha17}}]\label{LI usually holds} 
    For a fixed odd prime power $q$ and fixed $g \geq 1$, \[\lim_{n \to \infty} \frac{\#\{C \in \mathcal{H}_{2g + 1, q^n} \mid \text{$C$ satisfies Linear Independence}\}}{\#\mathcal{H}_{2g + 1, q^n}} = 1.\] 
\end{prop}

\begin{prop}[Deligne's Equidistribution Theorem {\cite[Theorem 10.8.2]{KS99}}]\label{Deligne}
    Let $f$ be a continuous function on $\mathrm{USp}_{2g}(\mathbb{C})$, such that $f(U)$ is central (only dependent on the conjugacy class of $U$). Then for fixed genus $g \geq 1$, \[\lim_{n \to \infty} \frac{1}{\#\mathcal{H}_{2g + 1, q^n}}\sum_{C \in \mathcal{H}_{2g + 1, q^n}} f(\vartheta(C/\mathbb{F}_{q^n})) = \int_{\mathrm{USp}_{2g}(\mathbb{C})} f(U)\, d\mu_{\mathrm{Haar}}(U).\]
\end{prop}

Proposition \ref{LI usually holds} will allow us to use the explicit formula given by Proposition \ref{prop:kfree_bcfq} to analyze $B_k(C/\mathbb{F}_{q^n})$, since the proportion of cases where the formula does not hold will go to $0$ in the large $n$ limit. Meanwhile, Proposition \ref{Deligne} allows us to rephrase questions about function fields corresponding to hyperelliptic curves in terms of the Haar measure on $\mathrm{USp}_{2g}(\mathbb{C})$. We can reformulate Deligne's Equidistribution Theorem to more readily apply to our situation, as follows. 

\begin{cor} \label{Deligne restatement}
    For any central Borel set $A \subset \mathrm{USp}_{2g}(\mathbb{C})$ whose boundary has Haar measure zero, we have \[\lim_{n \to \infty} \frac{\#\{C \in \mathcal{H}_{2g + 1, q^n} \mid \vartheta(C/\mathbb{F}_{q^n}) \subset A\}}{\#\mathcal{H}_{2g + 1, q^n}} = \mu_{\mathrm{Haar}}(A).\] 
\end{cor}

\begin{proof}
    For each $n$, consider the probability measure $\operatorname{Prob}_{\mathcal{H}_{2g + 1, q^n}}$ on $\mathrm{USp}_{2g}(\mathbb{C})$ (which contains a point mass at each $\vartheta(C/\mathbb{F}_{q^n})$, scaled by $\frac{1}{\#\mathcal{H}_{2g + 1, q^n}}$). Then Deligne's Equidistribution Theorem states that the sequence $\operatorname{Prob}_{\mathcal{H}_{2g + 1, q^n}}$ converges weakly to $\mu_{\mathrm{Haar}}$. Applying the Portmanteau theorem gives the desired result.
\end{proof}

To analyze $B_k(C/\mathbb{F}_q)$ in terms of matrices $U \in \mathrm{USp}_{2g}(\mathbb{C})$, we must first express it in terms of $\vartheta(C/\mathbb{F}_q)$.
Define a function $\varphi : \mathrm{USp}_{2g}(\mathbb{C}) \to \mathbb{R}$, given by \[\varphi(U) = \sum_{j = 1}^{2g} \frac{1}{|\mathcal{Z}'(\theta_j)|},\] where the eigenvalues of $U$ are $e^{i\theta_1}$, \ldots, $e^{i\theta_{2g}}$. Note that $\varphi$ is central (only dependent on the conjugacy class of $U$). 

\begin{prop}\label{Bk to varphi}
    If $C/\mathbb{F}_q$ satisfies the Linear Independence hypothesis, then \[\frac{B_k(C/\mathbb{F}_q)}{q^{g - g/k - 1/2}} = \varphi(\vartheta(C/\mathbb{F}_q))(1 + O_g(q^{-1/2k})).\] 
\end{prop}

\begin{proof}
    By Proposition \ref{prop:kfree_bcfq}, we have
    \[B_k(C/\mathbb{F}_q) = \max_{0 \leq a \leq k - 1} \sum_{j = 1}^{2g} \left|c_{j, a}\right|\cdot \left|\frac{\gamma_j}{Z'(\gamma_j^{-1})}\right|.\]

    First, we will analyze the terms $c_{j, a}$. Consider \[\frac{Z(\gamma_{j, \ell}^{-1})}{\gamma_{j, \ell} - 1} = \frac{\prod_{m = 1}^{2g}(1 - \gamma_m\gamma_{j, \ell}^{-1})}{(1 - \gamma_{j, \ell}^{-1})(1 - q\gamma_{j, \ell}^{-1})(\gamma_{j, \ell} - 1)}.\] Scaling each term in the product to be of the form $1 - x$ for $|x| < 1$, we can rewrite this as 
    \begin{align*}
        \frac{1}{q} \cdot \prod_{m = 1}^{2g} \gamma_m \cdot  \gamma_{j, \ell}^{-2g}& \cdot \frac{\prod_{m = 1}^{2g}(1 - \gamma_m^{-1}\gamma_{j, \ell})}{(1 - \gamma_{j, \ell}^{-1})^2(1 - q^{-1}\gamma_{j, \ell})} \\ & = q^{g - g/k - 1} e^{-2gi(\theta(\gamma_j) + 2\pi \ell)/k}\cdot \frac{\prod_{m = 1}^{2g} (1 - \gamma_m^{-1}\gamma_{j, \ell})}{(1 - \gamma_{j, \ell}^{-1})^2(1 - q^{-1}\gamma_{j, \ell})}.
    \end{align*}
    Now we have 
    \begin{align*}
        1 - \gamma_m^{-1}\gamma_{j, \ell} &= 1 + O_g(q^{-1/2 + 1/2k}), \\
        \frac{1}{1 - \gamma_{j, \ell}^{-1}} &= 1 + O_g(q^{-1/2k}),\\
        \frac{1}{1 - q^{-1}\gamma_{j, \ell}} &= 1 + O_g(q^{-1 + 1/2k}).
    \end{align*}

    This means \[\frac{Z(\gamma_{j, \ell}^{-1})}{\gamma_{j, \ell} - 1} = q^{g - g/k - 1}e^{-2gi(\theta(\gamma_j) + 2\pi \ell)/k}(1 + O_g(q^{-1/2k})),\] and therefore \[c_{j, a} = q^{g - g/k - 1}e^{-2gi\theta(\gamma_j)/k}\cdot \sum_{\ell = 0}^{k - 1}e^{2\pi i \ell(a - 2g)/k}(1 + O_g(q^{-1/2k})).\] 
    
    Summing the roots of unity using the geometric series formula, we have \[\frac{1}{k}\sum_{\ell = 0}^{k - 1}e^{2\pi i \ell(a - 2g)/k}(1 + O_g(q^{-1/2k})) = \begin{cases} 1 + O_g(q^{-1/2k}) & \text{if } k \mid a - 2g \\ O_g(q^{-1/2k}) & \text{otherwise}.\end{cases}\] Thus, for large $q$, the maximum value of each $|c_{j, a}|$ occurs when $a \equiv 2g \pmod{k}$, and we get \[B_k(C/\mathbb{F}_q) = q^{g - g/k - 1}\sum_{j = 1}^{2g} \left|\frac{\gamma_j}{Z'(\gamma_j^{-1})}\right|(1 + O_g(q^{-1/2k})).\] Finally, we have \[\frac{\gamma_j}{Z'(\gamma_j^{-1})} = \frac{\gamma_j(1 - \gamma_j^{-1})(1 - q\gamma_j^{-1})}{-\gamma_j\prod_{m \neq j}(1 - \gamma_m\gamma_j^{-1})} = q\gamma_j\cdot \frac{(1 - \gamma_j^{-1})(1 - q^{-1}\gamma_j)}{\mathcal{Z}_{\vartheta(C/\mathbb{F}_q)}'(\theta_j)}.\] We know $|q\gamma_j| = q^{1/2}$, while $(1 - \gamma_j^{-1})(1 - q^{-1}\gamma_j) = 1 + O_g(q^{-1/2})$, so summing over all $j$ gives \[B_k(C/\mathbb{F}_q) = q^{g - g/k - 1/2}\cdot \varphi(\vartheta(C/\mathbb{F}_q))\cdot (1 + O_g(q^{-1/2k})),\] as desired. 
\end{proof}

Since Proposition \ref{Bk to varphi} shows that $B_k(C/\mathbb{F}_q)$ is on the order of $q^{g - g/k - 1/2}$, we will instead work with the normalized bound \[\widetilde{B}_k(C/\mathbb{F}_q) := \frac{B_k(C/\mathbb{F}_q)}{q^{g - g/k - 1/2}}.\]
In a similar vein to the work of Humphries \cite{Hum13}, we analyze the average behavior of $\widetilde{B}_k(C/\mathbb{F}_{q^n})$ across the family of hyperelliptic curves by applying Deligne's Equidistribution Theorem to $\varphi$. 

First, we give an explicit description of the Haar measure on $\mathrm{USp}_{2g}(\mathbb{C})$. 

\begin{prop}[Weyl Integration Formula {\cite[Sec 5.0.4]{KS99}}]
    For any bounded, Borel-measurable central function $f : \mathrm{USp}_{2g}(\mathbb{C}) \to \mathbb{C}$, we have \[\int_{\mathrm{USp}_{2g}(\mathbb{C})} f(U)\, d\mu_{\mathrm{Haar}}(U) = \int_0^{\pi}\cdots \int_0^{\pi} f(\theta_1, \ldots,\theta_g)d\mu_{\mathrm{USp}}(\theta_1, \ldots, \theta_g),\] where \[d\mu_{\mathrm{USp}}(\theta_1, \ldots, \theta_g) = \frac{2^{g^2}}{g!\pi^g}\prod_{1 \leq m < n \leq g} (\cos \theta_n - \cos \theta_m)^2\prod_{\ell = 1}^g \sin^2 \theta_\ell\, d\theta_1\cdots d\theta_g.\] 
\end{prop}

Using this, we can prove the following fact about $\varphi$, which will allow us to use the restatement of Deligne's Equidistribution Theorem in Corollary \ref{Deligne restatement} to rephrase our problem about the average behavior of $\widetilde{B}_k(C/\mathbb{F}_{q^n})$ to a question about the Haar measure on $\mathrm{USp}_{2g}(\mathbb{C})$. 

For an interval $A \subset \mathbb{R}$, we will use the notation \[\{\varphi(U) \in A\} := \{U \in \mathrm{USp}_{2g}(\mathbb{C}) \mid \varphi(U) \in A\}.\] 
\begin{lem}[{\cite[Lemma 2.8]{Hum13}}]
    For an interval $A \subset \mathbb{R}$, the boundary of the set $\{\varphi(U) \in A\}$ 
    has Haar measure zero. 
\end{lem}

\begin{proof}
    Using the Weyl Integration Formula, it suffices to show that the boundary of \[\{(\theta_1, \ldots, \theta_g) \in [0, \pi]^g \mid \varphi(\theta_1, \ldots, \theta_g) \in A\}\] has $\mu_{\mathrm{USp}}$-measure zero. 
    
    First, the set of $(\theta_1, \ldots, \theta_g) \in [0, \pi]^g$ where two angles are equal has Lebesgue measure zero, and therefore $\mu_{\mathrm{USp}}$-measure zero as well (since $\mu_{\mathrm{USp}}$ is absolutely continuous with respect to the Lebesgue measure). Thus, it suffices to consider sets of $(\theta_1, \ldots, \theta_g) \in [0, \pi]^g$ for which the $\theta_j$ are distinct and occur in a fixed permutation. On such a set, $\varphi$ is continuous, real analytic, and non-uniformly constant. Therefore, the subset with $\varphi(\theta_1, \ldots, \theta_g) = a$ has Lebesgue measure zero, and $\mu_{\mathrm{USp}}$-measure zero as well. 
\end{proof}

We now introduce the explicit correspondence between the behavior of the bound $\widetilde{B}_k(C/\mathbb{F}_{q^n})$ and the Haar measure on $\mathrm{USp}_{2g}(\mathbb{C})$, through a careful application of Deligne's Equidistribution Theorem to the function $\varphi$.

\begin{lem}\label{Bk and haar}
    For any $\beta > 0$, we have \[\lim_{n \to \infty} \frac{\#\{C \in \mathcal{H}_{2g + 1, q^n} \mid \widetilde{B}_k(C/\mathbb{F}_{q^n}) \leq \beta\}}{\# \mathcal{H}_{2g + 1, q^n}} = \mu_{\mathrm{Haar}}(\{\varphi(U) \leq \beta\}).\] 
\end{lem}

\begin{proof}
    We follow the proof in Humphries \cite[Proposition 3.1]{Hum13}. Consider \[\lim_{n \to \infty}\left(\frac{\#\{C \in \mathcal{H}_{2g + 1, q^n} \mid \widetilde{B}_k(C/\mathbb{F}_{q^n}) \leq \beta\}}{\#\mathcal{H}_{2g + 1, q^n}} - \frac{\#\{C \in \mathcal{H}_{2g + 1, q^n} \mid \varphi(\vartheta(C/\mathbb{F}_{q^n})) \leq \beta\}}{\#\mathcal{H}_{2g + 1, q^n}}\right).\] The limit of the second term is $\mu_{\mathrm{Haar}}(\varphi(U) \leq \beta)$ by Deligne's Equidistribution Theorem, so it suffices to show that the limit of the difference is $0$. 

    Fix some $\varepsilon > 0$. Then it suffices to show that the proportion of curves $C \in \mathcal{H}_{2g + 1, q^n}$ for which exactly one of $\widetilde{B}_k(C/\mathbb{F}_{q^n}) \leq \beta$ or $\varphi(\vartheta(C/\mathbb{F}_{q^n})) \leq \beta$ holds goes to $0$ in the limit. 
    Any such curve must belong to one of the following four cases:
    \begin{enumerate}
        \item $C/\mathbb{F}_{q^n}$ does not satisfy the Linear Independence hypothesis. 
        \item $\varphi(\vartheta(C/\mathbb{F}_{q^n})) \in (\beta - \varepsilon, \beta + \varepsilon)$.
        \item $C/\mathbb{F}_{q^n}$ satisfies the Linear Independence hypothesis, $\widetilde{B}_k(C/\mathbb{F}_{q^n}) \leq \beta$, and $\varphi(\vartheta(C/\mathbb{F}_{q^n})) \geq \beta + \varepsilon$. 
        \item $C/\mathbb{F}_{q^n}$ satisfies the Linear Independence hypothesis, $\widetilde{B}_k(C/\mathbb{F}_{q^n}) > \beta$, and $\varphi(\vartheta(C/\mathbb{F}_{q^n})) \leq \beta - \varepsilon$. 
    \end{enumerate}
    In the limit of large $n$, the proportion of curves in case (1) goes to zero by Proposition \ref{LI usually holds}. For cases (3) and (4), by Proposition \ref{Bk to varphi}, there is a constant $c$ such that \[(1 - cq^{-n/2k})\varphi(\vartheta(C/\mathbb{F}_{q^n})) \leq \widetilde{B}_k(C/\mathbb{F}_{q^n}) \leq (1 + cq^{-n/2k})\varphi(\vartheta(C/\mathbb{F}_{q^n})).\] In case (3) we would have \[(1 - cq^{-n/2k})(\beta + \varepsilon) \leq \beta,\] which is false for all sufficiently large $n$; similarly, in case (4) we would have \[\beta \leq (1 + cq^{-n/2k})(\beta - \varepsilon),\] again false for all sufficiently large $n$. Therefore, the proportion of curves belonging to cases (3) and (4) vanish as $n$ tends to infinity. 

    Finally, the limiting proportion of curves belonging to case (2) is given by
    \[\lim_{n \to \infty} \frac{\#\{C \in \mathcal{H}_{2g + 1, q^n} \mid \varphi(\vartheta(C/\mathbb{F}_{q^n})) \in (\beta - \varepsilon, \beta + \varepsilon)\}}{\#\mathcal{H}_{2g + 1, q^n}} = \mu_{\mathrm{Haar}}(\varphi(U) \in (\beta - \varepsilon, \beta + \varepsilon)).\] In the limit of small $\varepsilon$, this proportion then approaches $\mu_{\mathrm{Haar}}(\varphi(U) = \beta) = 0$.
\end{proof}

To draw conclusions from this correspondence, we have the following result about the size of $\varphi(U)$, which we can then translate into a result about the size of $\widetilde{B}_k(C/\mathbb{F}_q)$ via Lemma \ref{Bk and haar}.

\begin{prop}[{\cite[Proposition 3.2]{Hum13}}]\label{varphi ineq}
    For all $U \in \mathrm{USp}_{2g}(\mathbb{C})$, we have $\varphi(U) \geq 1$. Equality holds if and only if $(\theta_1, \ldots, \theta_g)$ is a permutation of $(\frac{\pi}{2g}, \frac{3\pi}{2g}, \ldots, \frac{(2g - 1)\pi}{2g})$. 
\end{prop}

 {
\renewcommand{\thethm}{\ref{intro_thm:kfree_global}}
\begin{thm}
    For any fixed $\beta > 0$, the quantity \[\lim_{n \to \infty} \frac{\#\{C \in \mathcal{H}_{2g + 1, q^n} \mid \widetilde{B}_k(C/\mathbb{F}_{q^n}) \leq \beta\}}{\#\mathcal{H}_{2g + 1, q^n}} \] is $0$ if $0 < \beta \leq 1$, and strictly between $0$ and $1$ if $\beta > 1$. 
\end{thm}
\addtocounter{thm}{-1}
}

\begin{proof}
    We follow the proof in Humphries \cite[Theorem 1.4]{Hum13}.
    
    By Lemma \ref{Bk and haar}, we have \[\lim_{n \to \infty} \frac{\#\{C \in \mathcal{H}_{2g + 1, q^n} \mid \widetilde{B}_k(C/\mathbb{F}_{q^n})\leq \beta\}}{\#\mathcal{H}_{2g + 1, q^n}} = \mu_{\mathrm{Haar}}(\varphi(U) \leq \beta).\] By Lemma \ref{varphi ineq}, we have $\varphi(U) \geq 1$ for all $U \in \mathrm{USp}_{2g}(\mathbb{C})$, and the set of $U$ for which $\varphi(U) = 1$ has Haar measure zero. This immediately implies that if $\beta \leq 1$, then \[\mu_{\mathrm{Haar}}(\varphi(U) \leq \beta) = 0.\] 
    Meanwhile, if $\beta > 1$, then we can find an open neighborhood (around a point with $\varphi(U) = 1$) for which $\varphi(U) \leq \beta$, which must have positive measure. Conversely, we can also find an open neighborhood with positive measure (around a point where two angles are equal) for which $\varphi(U) > \beta$. Therefore, we conclude $0 < \mu_{\mathrm{Haar}}(\varphi(U) \leq \beta) < 1$. 
\end{proof}

Finally, we have seen that for a function field $C/\mathbb{F}_q$ with $q$ large, the bound on $\ET_k(X)$ is on the order of $q^{g - g/k - 1/2}$. However, as suggested by the proof of Proposition \ref{Bk to varphi}, the bounds on $\ET_k(X)$ will be significantly greater for certain residue classes $X$ modulo $k$ --- different residue classes will have different fluctuations in their normalized error term. We proceed to study the variations in the error bounds across different residue classes modulo $k$. 

For each $0 \leq a \leq k-1$, define \[B_k(C/\mathbb{F}_q, a) := \limsup_{Y \to \infty}|\ET(a + kY)|.\] Then, by the same argument as in Proposition \ref{prop:kfree_bcfq}, we have \[B_k(C/\mathbb{F}_q, a) = \sum_{j = 1}^{2g}|c_{j, a}|\cdot \left|\frac{\gamma_j}{Z'(\gamma_j^{-1})}\right|.\] 
\begin{prop}\label{X mod k bound}
    Suppose $C/\mathbb{F}_q$ satisfies the Linear Independence hypothesis. Then if $a - 2g \equiv b \pmod{k}$ for some $0 \leq b \leq k - 2$, we have \[\frac{B_k(C/\mathbb{F}_q, a)}{(b + 1)q^{g - g/k - 1/2 - b/2k}} = \varphi(\vartheta(C/\mathbb{F}_q))(1 + O_g(q^{-1/2k})).\] 
\end{prop}
    
\begin{proof}
    Following the proof of Proposition \ref{Bk to varphi}, it suffices to analyze \[c_{j, a} = \frac{1}{k}\sum_{\ell = 0}^{k - 1} \frac{Z(\gamma_{j, \ell}^{-1})}{\gamma_{j, \ell} - 1}\cdot e^{2\pi i \ell a/k}.\] Again, expand \[\frac{Z(\gamma_{j, \ell}^{-1})}{\gamma_{j, \ell} - 1} = q^{g - g/k - 1}e^{-2gi(\theta(\gamma_j) + 2\pi \ell)/k}\cdot \frac{\prod_{m = 1}^{2g}(1 - \gamma_m^{-1}\gamma_{j, \ell})}{(1 - \gamma_{j, \ell}^{-1})^2(1 - q^{-1}\gamma_{j, \ell})}.\] Then, it suffices to show that \[\frac{1}{k}\sum_{\ell = 0}^{k - 1} e^{2 \pi i(a - 2g)\ell/k} \cdot \frac{\prod_{m = 1}^{2g}(1 - \gamma_m^{-1}\gamma_{j, \ell})}{(1 - \gamma_{j, \ell}^{-1})^2(1 - q^{-1}\gamma_{j, \ell})} = (b + 1)q^{g - g/k - 1/2 - b/2k}(1 + O(q^{- 1/2k})).\] 
    We study each summand. Let $r = q^{-1/2k}$, so that \[\gamma_m^{-1}\gamma_{j, \ell} = r^{k - 1}\omega_m\] where $\omega_m := e^{i(-\theta(\gamma_m) + (\theta(\gamma_j) + 2\pi \ell)/k)}$ has absolute value $1$ for all $1 \leq m \leq 2g$, \[1 - \gamma_{j, \ell}^{-1} = 1 - r\omega_0\] where $\omega_0 := e^{-i(\theta(\gamma_j) + 2\pi \ell)/k}$ has absolute value $1$, and \[1 - q^{-1}\gamma_{j, \ell} = 1 - r^{2k - 1}\omega_{2g + 1}\] where $\omega_{2g + 1} := e^{i(\theta(\gamma_j) + 2\pi \ell)/k}$ has absolute value $1$ as well. Using geometric series expansion for the terms in the denominator, we have \[\frac{\prod_{m = 1}^{2g}(1 - \gamma_m^{-1}\gamma_{j, \ell})}{(1 - \gamma_{j, \ell}^{-1})^2(1 - q^{-1}\gamma_{j, \ell})} = \sum_{d = 0}^{k - 2} (d + 1)\omega_0^dr^d + \left(k\omega_0^{k - 1} - \sum_{m = 1}^{2g} \omega_m\right)r^{k - 1} + O(r^k).\] 
    
    The coefficient of the $r^d$ term in \[e^{2\pi i (a - 2g)\ell/k}\cdot \frac{\prod_{m = 1}^{2g}(1 - \gamma_m^{-1}\gamma_{j, \ell})}{(1 - \gamma_{j, \ell}^{-1})^2(1 - q^{-1}\gamma_{j, \ell})}\] is then \[(d + 1)e^{-i(d\theta(\gamma_j) + 2\pi \ell (d + 2g - a))/k}\] for $0 \leq d \leq k - 2$.
    Now, we sum across all $\ell$ and use the fact that \[\frac{1}{k}\sum_{\ell = 0}^{k - 1} e^{2\pi i \ell t/k} = \begin{cases} 1 & \text{if $k \mid t$} \\ 0 & \text{otherwise}.\end{cases}\] The first term that does not vanish will have $d = b$ (as we need $k \mid d + 2g - a$), which gives the desired result.  
\end{proof}

Then, the same analysis that we performed earlier on the maximal bound $B_k(C/\mathbb{F}_{q^n})$ on hyperelliptic curves can be used to study $B_k(C/\mathbb{F}_{q^n}, a)$ in the limit of large $n$, with a normalization given by Proposition \ref{X mod k bound}. For brevity and for $0 \leq b\leq k-2$, we define \[\widetilde{B}_k(C/\mathbb{F}_{q}, b + 2g) := \limsup_{Y \to \infty} \frac{\ET(b+2g+kY)}{(b + 1)q^{(g - g/k - 1/2 - b/2k)}},\] giving us the following theorem.

{
\renewcommand{\thethm}{\ref{intro_thm:kfree_xmodk}}
\begin{thm}
    For any $0 \leq b \leq k - 2$, we have \[\lim_{n \to \infty} \frac{\#\{C \in \mathcal{H}_{2g + 1, q^n} \mid \widetilde{B}_k(C/\mathbb{F}_{q^n}, b + 2g) \leq \beta\}}{\#\mathcal{H}_{2g + 1, q^n}}\] is $0$ if $0 < \beta \leq 1$, and strictly between $0$ and $1$ if $\beta > 1$. 
\end{thm}
\addtocounter{thm}{-1}
}

On the other hand, the case $b = k - 1$ behaves differently --- in the proof of Proposition \ref{X mod k bound}, the first term that does not vanish is still $r^{k - 1}$. However, in this case, its coefficient is \[k\omega_0^{k - 1} - \sum_{m = 1}^{2g}\omega_m.\] Using the definitions of $\omega_0$ and $\omega_m$, we can calculate that this coefficient is \[ke^{-i(k - 1)\theta(\gamma_j)/k} - xe^{i(\theta(\gamma_j)/k},\] where $x = \sum_{m = 1}^{2g} e^{-i\theta(\gamma_m)}$ is real. Since its absolute value $|k - xe^{i\theta(\gamma_j)}|$ depends on $\theta(\gamma_j)$ as well as the curve $C$, the bound $B_k(C/\mathbb{F}_q, k - 1 + 2g)$ does not directly relate to the function $\varphi$. It is possible to define a different function \[\varphi^*(U) = \sum_{j = 1}^{2g} \frac{\left|k - e^{i\theta_j}\sum_{m = 1}^{2g} e^{-i \theta_m}\right|}{|\mathcal{Z}'(\theta_j)|}\] and analyze its properties on $\mathrm{USp}_{2g}(\mathbb{C})$ instead. We will not perform this analysis, but it does suggest that the error term $B_k(C/\mathbb{F}_q, k - 1 + 2g)$ is still on the order of $q^{g - g/k - 1/2 - (k - 1)/2k}$.

\section{Summatory Function of the Totient Function}\label{sec:totient}

In this section, we study the limiting behavior of the summatory function \[F_\Phi(X) := \sum_{\substack{D \geq 0\\ 0 \leq \deg(D) < X}} \Phi(D).\] Similar to our results on $k$-free divisors in the previous section, we begin in Section \ref{subsec:error_term_totient} by writing $F_\Phi(X)$ in terms of a \textit{main term} $\mathrm{MT}_\Phi(X)$ and an \textit{error term} $R_\Phi(X)$. In the case where $\zeta_{C/\F_q}(s)$ has simple zeros, we proceed to derive an explicit expression for the \textit{normalized} error term $\widetilde{R}_\Phi(X)$ in terms of these zeros. Interestingly, as mentioned in Section \ref{sec:intro}, the analogous error term for the summatory function of the Euler totient function in the \textit{classical} setting cannot be written in terms of a sum of zeros of the classical Riemann zeta function: instead, the classical error term splits into a sum of an \textit{analytic} error term involving the zeros of the zeta function and an \textit{arithmetic} error term \cite{KW10}. Therefore, our work in this section provides an example where the properties of analogues of arithmetic functions over function fields differ from their number theoretic properties in the classical setting. In Section \ref{subsec:multiplezeros_totient}, we prove an analogous result to the main result of Section \ref{subsec:multiplezeros_kfree} by demonstrating that \[R_\Phi(X) = O(X^{r-1}q^{X/2}), \quad \limsup_{X \to \infty} \frac{|R_k(X)|}{X^{r-1}q^{X/2}} > 0,\] where $r$ is the maximal order of a zero of $\zeta_{C/\F_q}(s)$. In Section \ref{subsec:limiting_distribution_totient}, we parallel the work in Section \ref{subsec:limiting_distribution_kfree} to show that $\ET_\Phi(X)$ has a limiting distribution whenever the zeros of $\zeta_{C/\F_q}(s)$ are simple. Under the stronger assumption of Linear Independence, we can explicitly determine the natural density of the subsets \[\mathcal{S}_{\Phi}(\beta) := \{X \in \Z^+ \ |\ |\ET_\Phi(X)| \leq \beta\}\] for any $\beta \in \R^+$. Finally, in Section \ref{subsec:global_totient}, we derive an explicit bound on $\ET_\Phi(X)$ when $C/\F_q$ satisfies the Linear Independence hypothesis and proceed to study the behavior of $R_\Phi(X)$ on function fields of hyperelliptic curves over $\F_{q^n}$ in the limit of large $n$, as in Section \ref{subsec:global_kfree}. The proofs of many results in this section are simplified versions of their analogues in Section \ref{sec:kfree}, so they will either be omitted or sketched.

\subsection{Computation of the Error Term}\label{subsec:error_term_totient} In this subsection, we find an expression for $F_{\Phi}(X)$ by similar methods to those in Section \ref{subsec:error_term_kfree}. As before, we study the coefficients of the Dirichlet series \[D_{\Phi}(s) := \sum_{D \geq 0} \frac{\Phi(D)}{\mathcal{N}D^s}\] by comparing them to an expression in terms of the zeta function $\zeta_{C/\F_q}(s)$. The proof of the following lemma is identical to the proof of Lemma \ref{lem:kfree_dirichlet_zeta} and is therefore omitted.

\begin{lem}\label{lem:totient_dirichlet_zeta}
For a function field $C/\F_q$ with associated zeta function $\zeta_{C/\F_q}(s) = Z_{C/\F_q}(u)$, the Dirichlet series associated with $\Phi$ is given by \[D_\Phi(s) = \frac{\zeta_{C/\F_q}(s-1)}{\zeta_{C/\F_q}(s)} = \frac{Z_{C/\F_q}(qu)}{Z_{C/\F_q}(u)}.\]
\end{lem}
As in Section \ref{subsec:error_term_kfree}, we begin our work in the genus zero case (i.e., $C/\F_q = \F_q(T)$), where we will instead study the summatory function \[F_{\Phi,0}(X) := \sum_{\substack{f \text{ monic}\\ \deg(f) < X}} \Phi(f).\] By similar arguments to those made in Section \ref{subsec:error_term_kfree}, we derive the following results:

\begin{lem} \label{lem:genus_zero_zeta}
When $\Re(s) > 2$, we have the equality \[\sum_{f ~\text{monic}} \frac{\Phi(f)}{|f|^s} = \frac{\zeta_0(s - 1)}{\zeta_0(s)} = \frac{Z_0(qu)}{Z_0(u)}.\]
\end{lem}

\begin{prop} \label{prop:genus_zero_explicit_totient} 
For any positive integer $X$, we have \[F_{\Phi,0}(X) = \frac{q^{2X - 1} + 1}{q + 1}.\]
\end{prop}

We now return to the general case where $C/\F_q$ has genus $g \geq 1$ and consider the summatory function $F_\Phi(X)$ across all divisors of $C/\F_q$. We analyze this function in the same manner as we did for $Q_k(X)$ in Section \ref{subsec:error_term_kfree}. In particular, under the assumption that $\zeta_{C/\F_q}$ has only simple zeros, we apply Cauchy's residue theorem and Lemma \ref{lem:totient_dirichlet_zeta} to find an expression for $F_\Phi(X)$ in terms of a main term and an error term, similar to the methods used by \cite{Cha17} and \cite{Hum12} in their studies of the summatory function of the M\"obius function. As before, we let $\zeta(s) = Z(u)$ denote the zeta function of $C/\F_q$ for brevity. Moreover, index the inverse zeros of $Z(u)$ so that $\gamma_{j+g} = \overline{\gamma_j}$ for $j = 1,\ldots,g$.

\begin{prop}\label{prop:simple_zeros_phi} Let $g \geq 1$, and suppose the zeros of $Z(u)$ are all simple. Then as $X$ tends to infinity, 
\[
\ET_\Phi(X) := \frac{F_{\Phi}(X) - \MT_\Phi(X)}{q^{X/2}}= - \sum_{j=1}^{2g} \frac{Z(q\gamma_j^{-1})}{Z'(\gamma_j^{-1})} \frac{\gamma_j}{\gamma_j - 1} e^{i \theta(\gamma_j)X} + O_{q,g}\left( \frac{1}{q^{X/2}} \right),
\]
where 
\[\MT_\Phi(X) = \frac{q^{1-g} h}{\zeta(2)(1 - q^{-1})(q^2 - 1)}q^{2X}
\]
and $h$ denotes the class number of the function field $C/\mathbb{F}_q$.
\end{prop}

\begin{proof}
Similar to the proof of Proposition \ref{prop:kfree_simple_zeros}, we study the contour integral \[\frac{1}{2\pi i}\int_{C_\rho} \frac{Z(qu)}{Z(u)u^{N+1}} \, du,\] where $N$ is a nonnegative integer, and $C_\rho$ is the circle of radius $\rho > 1$ oriented counterclockwise. Note that the integrand has a pole of order $N+1$ at $u = 0$, simple poles at each zero $\gamma^{-1}$ of $Z(u)$, and a simple pole at $u = q^{-2}$. By summing over the residues of the integrand at these poles and then summing across $0 \leq N < X$, we deduce that \begin{align*}
F_{\Phi}(X) = \frac{q^{1-g}h}{\zeta(2)(1-q^{-1})(q^2-1)}q^{2X}  - \sum_{j=1}^{2g} \frac{Z(q \gamma_j^{-1})\gamma_j}{Z'(\gamma_j^{-1})} \frac{\gamma_j^X}{\gamma_j - 1} + \varepsilon_{\Phi}(X),
\end{align*}

where 
\[
\varepsilon_\Phi(X) = - \frac{q^{1-g}h}{\zeta(2)(1-q^{-1})(q^2-1)} + \sum_{j=1}^{2g} \frac{Z(q \gamma_j^{-1})}{Z'(\gamma_j^{-1})} \frac{\gamma_j}{\gamma_j - 1} + \frac{1}{2 \pi i} \sum_{N = 0}^{X-1}\oint_{C_\rho} \frac{1}{u^{N+1}}\frac{Z(qu)}{Z(u)} \,du.
\] When $N \geq 1$, the integral \[\oint_{C_\rho} \frac{1}{u^{N+1}}\frac{Z(qu)}{Z(u)} \,du\] vanishes as $\rho$ approaches infinity. This implies $\varepsilon_\Phi(X)$ is a constant (not dependent on $X$), which can be calculated by setting $X = 1$ and using the fact that $F_\Phi(1) = 1$. 
The desired result follows upon dividing through by $q^{X/2}$. 
\end{proof}

As noted earlier, the normalized error term $\ET_\Phi(X)$ can be expressed (up to a small correction term) in terms of a sum of zeros of $Z(u)$. This result is markedly different from the case of the summatory function of the classical Euler totient function. In the classical setting, the error term has two components: an analytic error term that is expressed in terms of zeros of the Riemann zeta function, and an arithmetic error term \cite{KW10}. On the other hand, our result in Proposition \ref{prop:simple_zeros_phi} demonstrates that the arithmetic error term does \textit{not} appear in the function field setting.

We will now proceed to study the case where $Z(u)$ has zeros of multiple order. Similar to Section \ref{subsec:multiplezeros_kfree}, we will show that \[R_\Phi(X) := F_\Phi(X) - \mathrm{MT}_\Phi(X)\] 
has order of growth $X^{r-1}q^{X/2}$, where $r$ is the maximal order of a zero of $Z(u)$.

\subsection{Zeros of Multiple Order}\label{subsec:multiplezeros_totient}

Now that we have computed the error term under the assumption of simple zeros of $Z(u)$, it remains to consider the case of zeros of multiple order. As in Section \ref{subsec:multiplezeros_kfree}, we will show that the error term $R_\Phi(X)$ is bounded above and below by $X^{r-1}q^{X/2}$ whenever the maximal order of a zero of $\zeta_{C/\F_q}(s)$ is $r$. The appearance of the $q^{X/2}$ term in place of $q^{X/2k}$ results from the fact that the inverse zeros of $Z(u)$ have absolute value $q^{1/2}$, while the inverse zeros of $Z(u^k)$ have absolute value $q^{1/2k}$.

We first deal with the upper bound. As in Section \ref{subsec:multiplezeros_kfree}, the upper bound follows easily from the residue calculation in Lemma \ref{lem:residue_calculation}. 

\begin{prop}\label{prop:totient_bound_above} 
    If all zeros of $Z(u)$ have order at most $r$, then
    \[|R_\Phi(X)| = O(X^{r - 1}q^{X/2}).\] 
\end{prop}
\begin{proof}
    The proof is identical to the proof of Proposition \ref{prop:kfree_bound_above}.
\end{proof}

For the lower bound, we again consider two cases depending on whether $\sqrt{q}$ is an inverse zero of $Z(u)$. The following lemma is the analogue of Lemma \ref{cor2.9analogue} in Section \ref{subsec:multiplezeros_kfree} and is used in the proof of both cases.

\begin{lem}\label{cor2.9analogue_totient}
    For $|u| < q^{-2}$,
    \[
    \sum_{X=1}^\infty F_\Phi(X) u^{X-1} = \frac{Z(qu)}{Z(u)(1-u)}.
    \]
\end{lem}
The lower bound can be obtained using this together with Lemmas \ref{Vivanti--Pringsheim} and \ref{combinatorial_lem}. We first consider the case where $\sqrt{q}$ is an inverse zero of $Z(u)$.

\begin{prop}\label{prop:totient_sqrtqcase}
    Let $g \geq 1$, and suppose that $\sqrt{q}$ is an inverse zero of $Z(u)$ with order $r \geq 2$. Then 
    \[
    \limsup_{X \to \infty} \frac{|R_\Phi(X)|}{X^{r - 1}q^{X/2}} > 0.
    \] 
\end{prop}
\begin{proof}
    The proof is very similar to the proof of Proposition \ref{prop:kfree_sqrtqcase}. The two cases depend on the sign of 
    \begin{align}\label{totient:signcondition}
        \frac{(-1)^rZ(q^{1/2})}{Z^{(r)}(q^{-1/2})(1 - q^{-1/2})}. 
    \end{align}
    The main idea is to consider the power series
    \[\sum_{X \geq 1} (R_\Phi(X) + cX^{r - 1}q^{X/2})u^{X - 1}\]
    and use Lemmas \ref{combinatorial_lem} and \ref{cor2.9analogue_totient} to find an expression equal to the power series which converges to a holomorphic function on $|u| < q^{-1/2}$ by Lemma \ref{Vivanti--Pringsheim}. The sign condition on (\ref{totient:signcondition}) is then used to obtain bounds on $c$ and therefore on the limit supremum. 
\end{proof}

Again, the case where $\sqrt{q}$ is not an inverse zero of maximal order follows a similar (though slightly more complicated) argument. 

\begin{prop}\label{prop:totient_rmax}
    Let $g \geq 1$. Suppose that $\gamma = \sqrt{q}e^{i \theta(\gamma_j)}$ is an inverse zero of $Z(u)$ of order $r \geq 2$, and that the order of the inverse zero at $\sqrt{q}$ is strictly less than $r$ (it may be 0). Then 
    \[
    \limsup_{X \to \infty} \frac{|R_\Phi(X)|}{X^{r - 1}q^{X/2}} > 0.
    \] 
\end{prop}

\subsection{Limiting Distribution of the Error
Term}\label{subsec:limiting_distribution_totient}

In this subsection, we will compute the natural density in the positive integers of the set \[\mathcal{S}_{\Phi}(\beta) := \{X \in \mathbb{Z}^+ \ |\ |\ET_\Phi(X)| \leq \beta\}\] for a real constant $\beta > 0$. Many of the proofs in this section are identical, if not simpler, to the analogous results in Section \ref{subsec:limiting_distribution_kfree}, so they will be omitted for brevity. As before, let $\gamma_1,\ldots,\gamma_{2g}$ be the inverse zeros of $Z(u)$, assumed to be simple and ordered such that $\gamma_{j+g} = \overline{\gamma_j}$ for $j =1,\ldots, g$ and the argument of $\gamma_j$ lies in $[0,\pi]$ for $j= 1,\ldots,g$. As we saw in Section \ref{subsec:error_term_totient}, we can write \[\ET_{\Phi}(X) = E_{M,\Phi}(X) + O_{q,g}(q^{-X/2}),\] where \[E_{M,\Phi}(X) := -\sum_{j=1}^{2g} \frac{Z(q\gamma_j^{-1})}{Z'(\gamma_j^{-1})}\frac{\gamma_j}{\gamma_j-1}e^{i X \theta(\gamma_j)}.\] As before, we follow the methods employed by Humphries \cite{Hum12}.

Consider the group homomorphism $\varphi: \Z \to \T^g$ given by $\varphi(X) = (\exp(iX\theta(\gamma_j)))_{j=1}^{g}$. If we let $G$ denote the topological closure of the subgroup $\varphi(\Z)$ in $\T^g$, then by the Kronecker--Weyl theorem, $G$ is a closed subgroup of $\T^g$.

\begin{prop}\label{prop:totient_main_error_distribution}
Let $C/\mathbb{F}_q$ be a function field of genus $g \geq 1$, and suppose that all zeros of $Z(u)$ are simple. Consider the function $\alpha_\Phi: \T^g \to \R$ defined by \[\alpha_\Phi(z_1,\ldots,z_j) = -\sum_{j=1}^{g} \left(\frac{Z(q\gamma_j^{-1})}{Z'(\gamma_j^{-1})}\frac{\gamma_jz_j}{\gamma_j-1} + \frac{Z\left(q\overline{\gamma_j^{-1}}\right)}{Z'\left(\overline{\gamma_j^{-1}}\right)}\frac{\overline{\gamma_jz_j}}{\overline{\gamma_j}-1}\right)\] and the pushforward probability measure $\nu_\Phi := (\alpha_\Phi)_*\mu_G$, where $\mu_G$ is the Haar measure on $G$. For every continuous function $f: \R \to \R$, we have \[\lim_{N \to \infty} \frac{1}{N}\sum_{X=1}^N f(E_{M,\Phi}(X)) = \int_\R f(x) \, d\nu_\Phi(x),\] i.e., $\nu_\Phi$ is the limiting distribution for $E_{M,\Phi}(X)$.
\end{prop}

\begin{proof}
Notice that $\alpha_\Phi \circ \varphi(X) = E_{M, \Phi}(X)$ for each positive integer $X$. The rest of the proof follows similarly to the proof of Proposition \ref{prop:kfree_main_error_distribution}.
\end{proof}

\begin{prop}\label{prop:totient_limiting_distribution}
Let $C/\mathbb{F}_q$ be a function field of genus $g \geq 1$, and suppose that all zeros of $Z(u)$ are simple. The overall error term $\ET_\Phi(X)$ has limiting distribution $\nu_\Phi = (\alpha_\Phi)_*\mu_G$. That is, for every continuous function $f: \R \to \R$, we have \[\lim_{N \to \infty} \frac{1}{N}\sum_{X=1}^N f(\ET_\Phi(X)) = \int_\R f(x) \, d\nu_\Phi(x).\] 
\end{prop}

\begin{proof}
The proof is very similar to the proof of Proposition \ref{prop:kfree_limiting_distribution}.
\end{proof}
We will now explicitly compute the natural density of
\[\mathcal{S}_{\Phi}(\beta) := \{X \in \mathbb{Z}^+ \ |\ |\ET_{\Phi}(X)| \leq \beta\}\] in the set of positive integers. As before, we will need to assume the Linear Independence hypothesis for our result in order to apply the Kronecker--Weyl theorem.

{
\renewcommand{\thethm}{\ref{intro_thm:totient_lim_dist}}
\begin{thm}
Assume that $C/\F_q$ satisfies the Linear Independence hypothesis. The natural density of the set $\mathcal{S}_\Phi(\beta)$, denoted $\delta(\mathcal{S}_\Phi(\beta))$, exists and is equal to \[m\left(\sum_{j=1}^g 2\left|\frac{Z(q\gamma_j^{-1})}{Z'(\gamma_j^{-1})}\frac{\gamma_j}{\gamma_j-1}\right|\cos(\theta_j) \in [-\beta,\beta]\right) ,\]
where $m$ is the Lebesgue measure on $[0,2\pi)^g$.
\end{thm}
\addtocounter{thm}{-1}
}

\begin{proof}
Due to the Linear Independence hypothesis, it follows by the Kronecker--Weyl theorem that the topological closure of $G$ in $\mathbb{T}^g$ is the entirety of $\mathbb{T}^g$. Thus, the normalized Haar measure on $G$ is the Lebesgue measure on the $g$-torus. The remainder of the proof proceeds as in the proof of Theorem \ref{intro_thm:kfree_limdist} in Section \ref{subsec:limiting_distribution_kfree}.
\end{proof}

By the same arguments that we made for Corollary \ref{intro_cor:kfree_no_bias}, we deduce that the normalized error term $\ET_\Phi(X)$ is unbiased.

{
\renewcommand{\thethm}{\ref{intro_cor:totient_no_bias}}
\begin{cor}
    Let $C/\mathbb{F}_q$ be a function field of genus $g \geq 1$, and suppose that $C$ satisfies the Linear Independence hypothesis. Then the natural densities of the sets \[S_\Phi^+ = \left\{X \in \Z^+ \ |\ \ET_\Phi(X) > 0 \right\} \quad \text{and} \quad S_\Phi^- = \left\{X \in \Z^+\ |\ \ET_\Phi(X) < 0 \right\}\] exist and are given by \[\delta(S_\Phi^+) = \delta(S_\Phi^-) = \frac12.\]
\end{cor}
\addtocounter{thm}{-1}
}

Similar to Section \ref{subsec:limiting_distribution_kfree}, we conclude by calculating the Fourier transform of the limiting distribution $\nu_\Phi$. As in the case of Proposition \ref{prop:fourier_kfree}, the proof of the following proposition is a routine calculation and is therefore omitted.

\begin{prop}\label{prop:fourier_totient}
Assume that $C/\F_q$ satisfies the Linear Independence hypothesis. The Fourier transform of $\nu_\Phi$ exists and is given by \[\widehat{\mu}_\Phi(y) = \prod_{j=1}^g J_0\left(2\left|\frac{\gamma_j}{\gamma_j - 1}\right|\left|\frac{Z(q\gamma_j^{-1})}{Z'(\gamma_j^{-1})}\right|y\right),\] where \[J_0(z) = \sum_{m=0}^\infty \frac{(-1)^m(z/2)^{2m}}{(m!)^2}\] is the Bessel function of the first kind.
\end{prop}

\subsection{Global Behavior of the Error Term on Families of Hyperelliptic Curves}\label{subsec:global_totient}
In the same vein as our work in Section \ref{subsec:global_kfree}, we will now analyze how the bound \[B_{\Phi}(C/\F_q) := \limsup_{X \to \infty} |\ET_\Phi(X)|\] behaves on average for function fields corresponding to hyperelliptic curves $C \in \mathcal{H}_{2g + 1, q^n}$ for fixed $g$ and $q$, in the limit $n$ approaches $\infty$. 

Under the assumption of the Linear Independence hypothesis, we first derive an explicit expression for $B_{\Phi}(C/\F_q)$. By using the fact that the set \[\left\{\left(e^{iX\theta(\gamma_1)},\ldots,e^{iX\theta(\gamma_g)}\right)\right\}_{X \in \Z}\] is dense in the $g$-torus $\T^g$ (thanks to the Kronecker--Weyl theorem), we can deduce the following proposition. Since its proof is quite similar in spirit to the proof of Proposition \ref{prop:kfree_bcfq} and the proof of \cite[Theorem 2.6]{Hum12}, we state the result without proof. 

\begin{prop}\label{prop:totient_bcfq}
Assume that $C/\F_q$ satisfies the Linear Independence hypothesis. Then, \[B_\Phi(C/\F_q) = \sum_{j=1}^{2g}\left|\frac{Z(q\gamma_j^{-1})}{Z'(\gamma_j^{-1})}\frac{\gamma_j}{\gamma_j-1}\right|.\]
\end{prop}

As before, we will work with the function \[\varphi(U) = \sum_{j = 0}^{2g} \frac{1}{|\mathcal{Z}'(\theta_j)|}\] to study the global behavior of $B_{\Phi}(C/\F_q)$ on families of hyperelliptic curves. Recall that $\vartheta(C/\mathbb{F}_q)$ denotes the unitarized Frobenius conjugacy class of $C/\mathbb{F}_q$. 
\begin{lem}\label{Bk to varphi totient}
    If $C/\F_q$ satisfies the Linear Independence hypothesis, then \[\frac{B_\Phi(C/\mathbb{F}_q)}{q^{2g - 2}} = \varphi(\vartheta(C/\mathbb{F}_q))(1 + O_g(q^{-1/2})).\] 
\end{lem}

\begin{proof}
    By Proposition \ref{prop:totient_bcfq}, we have \[B_\Phi(C/\mathbb{F}_q) = \sum_{j=1}^{2g}\left|\frac{Z(q\gamma_j^{-1})}{Z'(\gamma_j^{-1})}\frac{\gamma_j}{\gamma_j-1}\right|.\] Using the fact that $|\gamma_j| = q^{1/2}$, we can write
    \begin{align*}
        Z(q\gamma_j^{-1}) &= \frac{\prod_{m = 1}^{2g} (1 - q\gamma_j^{-1}\gamma_m)}{(1 - q\gamma_j^{-1})(1 - q^2\gamma_j^{-1})} \\
        &= q^{2g - 3} \gamma_j^{-2g + 2}\prod_{m = 1}^{2g} \gamma_m \cdot \left(\frac{\prod_{m = 1}^{2g} (1 - q^{-1}\gamma_j\gamma_m^{-1})}{(1 - q^{-1}\gamma_j)(1 - q^{-2}\gamma_j)}\right). 
    \end{align*}
    Since $|q^{-1}\gamma\gamma_j^{-1}| = q^{-1}$, $|q^{-1}\gamma| = q^{-1/2}$, and $|q^{-2}\gamma| = q^{-3/2}$, we have \[\frac{\prod_{m = 1}^{2g} (1 - q^{-1}\gamma_j\gamma_m^{-1})}{(1 - q^{-1}\gamma_j)(1 - q^{-2}\gamma_j)} = 1 + O_g(q^{-1/2}).\] Meanwhile, $|\gamma_j| = |\gamma_m| = q^{1/2}$, so we have \[|Z(q\gamma_j^{-1})| = q^{2g - 2}(1 + O_g(q^{-1/2})).\] On the other hand, \[\frac{\gamma_j}{Z'(\gamma_j^{-1})(\gamma_j - 1)} = -\frac{(1 - q\gamma_j^{-1})}{\gamma_j \prod_{m \neq j} (1 - \gamma_j^{-1}\gamma_m)} = \frac{q\gamma_j^{-2} - \gamma_j^{-1}}{\mathcal{Z}'(\theta(\gamma_j))}\] has absolute value $\frac{1}{|\mathcal{Z}'(\theta(\gamma_j))|}(1 + O_g(q^{-1/2}))$, so every term contributes \[\left|\frac{Z(q\gamma_j^{-1})}{Z'(\gamma_j^{-1})}\cdot \frac{\gamma_j}{\gamma_j - 1}\right| = \frac{q^{2g - 2}}{|\mathcal{Z}'(\theta(\gamma_j))|}(1 + O_g(q^{-1/2})).\] Summing over all inverse zeros gives the desired result. 
\end{proof}

In light of Proposition \ref{Bk to varphi totient}, we normalize the error bound to \[\widetilde{B}_\Phi(C/\mathbb{F}_q) := \frac{B_\Phi(C/\mathbb{F}_q)}{q^{2g - 2}}.\] Using the same facts about the behaviour of $\varphi$ and the correspondence between $\vartheta(C/\mathbb{F}_q)$ and $U \in \mathrm{USp}_{2g}(\mathbb{C})$ as that we established in Subsection \ref{subsec:global_kfree}, we have the following result:

\begin{lem}
    For any $\beta > 0$, we have \[\lim_{n \to \infty} \frac{\#\{C \in \mathcal{H}_{2g + 1, q^n} \mid \widetilde{B}_\Phi(C/\mathbb{F}_{q^n}) \leq \beta\}}{\#\mathcal{H}_{2g + 1, q^n}} = \mu_{\mathrm{Haar}}(\{\varphi(U) \leq \beta\}).\] 
\end{lem}

\begin{proof}
    The proof is the same as that of Lemma \ref{Bk and haar}. 
\end{proof}

Following the same logic as in Subsection \ref{subsec:global_kfree}, we deduce the following result.

{
\renewcommand{\thethm}{\ref{intro_thm:totient_global}}
\begin{thm}
    For any fixed $\beta > 0$, the quantity \[\lim_{n \to \infty} \frac{\#\{C \in \mathcal{H}_{2g + 1, q^n} \mid \widetilde{B}_\Phi(C/\mathbb{F}_{q^n}) \leq \beta\}}{\#\mathcal{H}_{2g + 1, q^n}}\] is $0$ if $0 < \beta \leq 1$, and strictly between $0$ and $1$ if $\beta > 1$. 
\end{thm}
\addtocounter{thm}{-1}
}

\nocite{*}
\bibliography{summatorykfreetotientfunctionfields}
\bibliographystyle{alpha}

\end{document}